\documentclass[amstex,10pt]{article}
\usepackage{amssymb,latexsym}
\usepackage{amsmath,amsthm}
\usepackage{graphicx}
\usepackage{subfigure}
\usepackage{caption}
\usepackage{mathtools}
\usepackage{epstopdf}
\usepackage{array,color,colortbl}
\usepackage{pdfpages}
\usepackage{dsfont}

\newcounter{case}
\newcounter{subcase}[case]
\DeclarePairedDelimiter\ceil{\lceil}{\rceil}
\DeclarePairedDelimiter\floor{\lfloor}{\rfloor}

\newtheorem{theorem}{Theorem}[section]
\newtheorem{lemma}[theorem]{Lemma}

\newtheorem{conjecture}[theorem]{Conjecture}

\newtheorem{corollary}[theorem]{Corollary}
\newtheorem{claim}{Claim}

\theoremstyle{remark}

\newtheorem{remark}[theorem]{Remark}
\newtheorem{definition}[theorem]{Definition}
\newtheorem{example}[theorem]{Example}

\renewcommand{\geq}{\geqslant}
\renewcommand{\leq}{\leqslant}
\renewcommand{\ge}{\geqslant}
\renewcommand{\le}{\leqslant}

\newcommand{\kn}{K_{n,n}}
\newcommand{\ekn}{E(K_{n,n})}
\newcommand{\sign}{\text{sign}}

\def\cref#1{Corollary~$\ref{#1}$}

\newcommand{\cc}{\mathcal{C}}

\newcommand{\ci}{\mathcal{I}}

\newcommand{\cf}{\mathcal F}

\newcommand{\cm}{\mathcal M}
\newcommand{\cn}{\mathcal N}

\newcommand{\cv}{\mathcal V}

\renewcommand{\geq}{\geqslant}
\renewcommand{\leq}{\leqslant}
\renewcommand{\ge}{\geqslant}
\renewcommand{\le}{\leqslant}

\def\cref#1{Corollary~$\ref{#1}$}

\title{Fair representation by independent sets}

\author{
Ron Aharoni \thanks{Research supported by  BSF grant no. 2006099, an ISF grant, and by the Discount Bank Chair at the Technion.}\\
\small Department of Mathematics, Technion, Haifa 32000, Israel\\
\small raharoni@gmail.com\\
\and Noga Alon \thanks{Research supported by a USA-Israeli
BSF grant, by an ISF grant, and by the Israeli I-Core program.}\\
\small Sackler School of Mathematics and Blavatnik School of
Computer Science, Tel Aviv University, Tel Aviv, Israel\\
\small nogaa@tau.ac.il\\
\and Eli Berger \thanks{Research supported by  BSF grant no. 2006099 and by an ISF grant.}\\
\small Department of Mathematics, Haifa University, Haifa 31999, Israel\\
\small berger@math.haifa.ac.il\\
\and Maria Chudnovsky \thanks{Research
partially supported by BSF Grant No. 2020345.}\\
\small Department of Mathematics, Princeton University, Princeton, NJ 08544, USA\\
\small mchudnov@math.princeton.edu\\
\and Dani Kotlar\\
\small Department of Computer Science, Tel-Hai College, Upper Galilee, Israel\\
\small dannykot@telhai.ac.il\\
\and Martin Loebl \thanks{Gratefully acknowledges support by the Czech Science Foundation GACR, contract number P202-13-21988S.}\\
\small Dept. of Applied Mathematics, Charles University, Praha, Malostranske n.25, 118 00 Praha 1 Czech Republic\\
\small loebl@kam.mff.cuni.cz\\
\and Ran Ziv\\
\small Department of Computer Science, Tel-Hai College, Upper Galilee, Israel\\
\small ranziv@telhai.ac.il\\
}

\date{}

\begin{document}

\maketitle

\begin{abstract}

For a hypergraph $H$ let $\beta(H)$ denote the minimal number of edges from $H$ covering $V(H)$. An edge $S$ of $H$  is said to represent {\em fairly} (resp. {\em almost fairly})  a partition  $(V_1,V_2, \ldots, V_m)$ of  $V(H)$
 if  $|S\cap V_i|\ge \floor*{\frac{|V_i|}{\beta(H)}}$ (resp. $|S\cap V_i|\ge \floor*{\frac{|V_i|}{\beta(H)}}-1$) for all $i \le m$.
 In matroids any partition of $V(H)$ can be represented  fairly by some independent set.
We look for classes of hypergraphs $H$ in which any partition of $V(H)$ can be represented almost fairly by some edge.
 We show that this is true when $H$ is the set of independent sets in a path, and conjecture that it is true when $H$ is the set of matchings in $K_{n,n}$. We  prove that partitions of $E(K_{n,n})$ into three sets can be represented almost fairly. The methods of proofs are topological.
\end{abstract}

\section{Introduction}

\subsection{Terminology and main theme}
A hypergraph $\cc$ is called a {\em simplicial complex} (or just a ``complex'') if it is  closed
down, namely $e \in \cc$ and $f \subseteq e$ imply $f \in \cc$. We denote by $V(\cc)$ the vertex set of $\cc$, and by $E(\cc)$ its edge set. Let  $\beta(\cc)$ be the minimal number of edges (``simplices'') of $\cc$ whose union is $V(\cc)$.
For any hypergraph $H$ we denote by $\Delta(H)$ the maximal degree of a vertex in $H$.

We say  that $S\in \cc$ represents a set $A$ of vertices {\em fairly} if $|S \cap A| \ge \floor*{ \frac{|A|}{\beta(\cc)}}$, and that it represents $A$ {\em almost fairly} if $|S \cap A| \ge \floor*{ \frac{|A|}{\beta(\cc)}}-1$.
We say that $S$ represents fairly (almost fairly) a collection of sets if it does so to each set in the collection, reminiscent of the way a parliament represents fairly the voters of the different parties.

Clearly, every set $A$ is fairly represented by some edge $S \in \cc$.
The aim of this paper is to study complexes $\cc$ in which for every partition $V_1, \ldots ,V_m$ of $V(\cc)$ there is an edge $S \in \cc$ representing all $V_i$'s fairly, or almost fairly.

In matroids, fair representation is always possible. The following can be proved, for example, by the use of Edmonds' matroids intersection theorem.

\begin{theorem}\label{2delta}
If $\cm$ is a matroid then  for every partition $V_1, \ldots ,V_m$ of $V(\cm)$ there exists a set $S\in \cm$ satisfying $|S\cap V_i|\ge \floor* {\frac{|V_i|}{\beta(\cc)}}$ for all $i$.
\end{theorem}

Classical examples which do not always admit fair representation are complexes of the form $\ci(G)$, the complex of independent sets in a graph $G$.
 In this case $\beta(\ci(G))=\chi(G)$, the chromatic number of $G$, which by Brooks' theorem is at most $\Delta(G)+1$.
  Indeed, there are classes of graphs for which the correct proportion of representation is $\frac{1}{\Delta(G)+1}$:

 \begin{theorem}\label{chordal}\cite{abz}
  If $G$ is chordal and $V_1, \ldots, V_m$ is a partition of its vertex set, then there exists an independent set
  of vertices $S$ such that $|S \cap V_i|\ge \frac{|V_i|}{\Delta(G)+1}$ for all $i \le m$.
   \end{theorem}

However, in general graphs this is not always true. The following theorem of Haxell \cite{penny} pinpoints the correct parameter.

\begin{theorem}\label{haxell}
If $\cv=(V_1,V_2, \ldots,V_m)$ is a partition of the vertex set of a graph $G$, and if $|V_i|\ge  2\Delta(G)$ for all $i \le m$, then there exists a set $S$, independent in $G$, intersecting all $V_i$'s.
\end{theorem}

This was an improvement over earlier results of Alon, who proved the
same with $25\Delta(G)$ \cite{Al1} and then with $2e\Delta(G)$ \cite{Al2}. The result is sharp, as shown in \cite{yuster,jin,szabotardos}.

\begin{corollary}
If the vertex set $V$ of a graph $G$ is partitioned into independent sets $V_1, V_2, \ldots ,V_m$ then there exists an independent  subset $S$ of $V$,
satisfying $|S \cap V_i| \ge \floor*{ \frac{|V_i|}{2\Delta(G)}}$ for every $i \le m$.
\end{corollary}

\begin{proof}
For each $i \le m$ let  $V^j_i$~($j \le \floor*{ \frac{|V_i|}{2\Delta(G)}}$) be disjoint subsets of size $2\Delta(G)$ of $V_i$. By Theorem \ref{haxell} there exists an independent set $S$ meeting all $V^j_i$, and this is the set desired in the theorem.
\end{proof}

\subsection{The special behavior of matching complexes}

Matching complexes, namely the independence complexes of line graphs, behave better than independence complexes of general graphs. For example, the following was proved   in \cite{aab}:

\begin{theorem}\label{lineimprovement}
If $G$ is the line graph of a graph and $V_1, \ldots ,V_m$ is a partition of $V(G)$ then
 there exists an independent set $S$ such that
 $|S \cap V_i| \ge \floor*{ \frac{|V_i|}{\Delta(G)+2}}$ for every $i\le m$.
\end{theorem}

This follows from a bound on the topological connectivity of the independence complexes of line graphs.
\begin{equation}
\label{eta1}
\eta(\ci(G))\ge \frac{|V|}{\Delta(G)+2}
\end{equation}

Here $\eta(\cc)$ is a connectivity parameter of the complex $\cc$ (for the definition see, e.g.,  \cite{aab}).
 The way from \eqref{eta1} to Theorem \ref{lineimprovement} goes
through a topological version of Hall's theorem, proved in \cite{ah}. A hypergraph version of \eqref{eta1} was proved in \cite{agn}. Theorem \ref{chordal} follows from the fact that if $G$ is chordal then $\eta(\ci(G))\ge \frac{|V|}{\Delta(G)+1}$.

So, matching complexes are more likely to admit fair representations.
We suggest four classes of complexes as candidates for having almost fair representation of disjoint sets.
\begin{enumerate}

\item The matching complex of a path.
\item
The matching complex of $K_{n,n}$.
\item The matching complex of any bipartite graph.
\item The intersection of two matroids.
\end{enumerate}

Since the third  class contains the first two and the fourth contains the third, conjecturing almost fair representation for them goes in ascending order of daring. In fact,
we only dare make the conjecture for the first two. As to the fourth, let us just remark that intersections of matroids often behave unexpectedly well with respect to partitions. For example, no instance is known to the authors in which, given two matroids $\cm$ and $\cn$, there holds $\beta(\cm \cap \cn)>\max(\beta(\cm), \beta(\cn))+1$.

\subsection{Independence complexes of paths}

In Section~\ref{paths} we prove that the independence complex of a path always admits almost fair representation. In fact, possibly more than that is true.  Since the matching complex
of a path is the independence complex of a path  one vertex shorter, a conjecture in this direction
(in a slightly stronger form) can be formulated as follows:

\begin{conjecture}\label{treesconj0}
Given a partition of the vertex set of a path into sets $V_1, \ldots, V_m$ there exists an independent set $S$ and integers $b_i,~i \le m$ , such that $|S \cap V_i| \ge \frac{|V_i|}{2}-b_i$ for all $i$, and

\begin{enumerate}
\item

$\sum_{i\le m}b_i \le \frac{m}{2}$

and
\item $b_i \le 1$ for all $i \le m$.
\end{enumerate}
\end{conjecture}

We prove the existence of sets satisfying either condition of Conjecture
\ref{treesconj0} (but not necessarily both simultaneously).

\begin{theorem}\label{pathscasetotal}
Given a partition of the vertex set of a path into sets $V_1, \ldots, V_m$ there exists an independent set $S$ and integers $b_i,~i \le m$ , such that $\sum_{i\le m}b_i \le \frac{m}{2}$ and $|S \cap V_i| \ge \frac{|V_i|}{2}-b_i$ for all $i$.
\end{theorem}

The proof of Theorem \ref{pathscasetotal}
uses the Borsuk-Ulam theorem.

\begin{theorem}\label{pathscaseindividual}
Given a partition of the vertex set of a cycle into sets $V_1, \ldots, V_m$ there exists an independent set $S$ such that $|S \cap V_i| \ge \frac{|V_i|}{2}-1$ for all $i$.
\end{theorem}

The proof
 uses a theorem of
Schrijver, strengthening a famous theorem of Lov\'asz on the
chromatic number of Kneser graphs. This means that it, too, uses indirectly the
Borsuk-Ulam theorem, since the Lov\'asz-Schrijver proof uses the
latter. We refer the reader to Matou{\v s}ek's book \cite{mat} for
background on topological methods in combinatorics, in particular
applications of the Borsuk-Ulam theorem.

\subsection{The matching complex of $K_{n,n}$}

\begin{conjecture}\label{equirep00}
For any partition $E_1, E_2, \ldots ,E_m$ of $E(K_{n,n})$ and any $j \le m$ there exists a perfect matching $F$
 in $K_{n,n}$ satisfying $|F \cap E_i| \ge \floor*{ \frac{|E_i|}{n} }$ for all $i \neq j$, and $|F \cap E_j| \ge \floor*{ \frac{|E_i|}{n} } -1$.
\end{conjecture}

We shall prove:

\begin{theorem}\label{}
Conjecture \ref{equirep00} is true for $m=2,3$.
\end{theorem}

For $m=2$ the result is simple, and the weight of the argument is in a characterization
of those cases in which
there necessarily exists an index $j$ for which $|M \cap E_j| \ge \floor*{ \frac{|E_j|}{n}}-1$.
The proof of the case $m=3$ is topological, using  Sperner's lemma.

\subsection{Relationship to known conjectures}\label{sec:known}

Conjecture \ref{treesconj0} is related to a well known conjecture of Ryser on Latin squares.
Given an $n \times n$ array $A$ of symbols, a {\em partial
transversal} is a  set of entries taken from
distinct rows and columns, and containing distinct symbols. A
partial transversal of size $n$ is called simply  a {\em
transversal}. Ryser's conjecture \cite{ryser} is that if $A$ is a Latin square, and $n$ is odd, then $A$ necessarily has a transversal.  The oddness condition is indeed necessary - for every even $n>0$ there exist $n \times n$ Latin squares
not possessing a transversal. An example is the
addition table of $\mathds{Z}_n$: if a transversal $T$ existed for this Latin square,
then the sum of its elements, modulo $n$, is  $\sum_{k \le n}k= \frac{n(n+1)}{2}\pmod n$.
On the other hand, since every row and every column is represented
in this sum,  the sum  is equal to
$\sum_{i\le n}i+\sum_{j \le n}j=n(n+1)\pmod n$,
and for $n$ even the two results do not agree.
Arsovski  \cite{arsovski} proved a closely related
conjecture, of Snevily,
that every square submatrix (whether even or odd) of the addition table of an odd order abelian group
 possesses a transversal.

Brualdi \cite{brualdi} and Stein {\cite{stein} conjectured that for any  $n$, any Latin square of order $n$  has a
partial transversal of order $n-1$.
Stein \cite{stein} observed that the same conclusion may follow from weaker conditions - the square does not have to be Latin, and
it may suffice that the entries of the $n \times n$ square are equally
distributed among $n$ symbols. Re-formulated, this becomes a special case of Conjecture \ref{equirep00}:

\begin{conjecture}\label{equirep0}
If   the edge set of $K_{n,n}$ is partitioned into
 sets $E_1, E_2, \ldots ,E_n$ of size $n$ each, then there exists a matching
 in $K_{n,n}$ consisting of one edge from all but possibly one
 $E_i$.
\end{conjecture}

Here, even for $n$ odd there are examples without a full transversal.
In  matrix terminology, take a matrix $M$ with $m_{i,j} =i$ for $j<n$,
$m_{i,n}=i+1$ for $i<n$, and $m_{n,n}=1$.

 A related conjecture to Conjecture \ref{equirep00} was suggested in \cite{awanlessbarat}:

\begin{conjecture} \label{strongest}
If $E_1, E_2, \ldots ,E_m$ are sets of edges in a bipartite graph, and
$|E_i| > \Delta(\bigcup_{i \le m} E_i)+1$ then there exists a rainbow matching.
\end{conjecture}

Re-phrased, this conjecture reads:
If $H$ is a bipartite multigraph, $G=L(H)$ and $V_i \subseteq V(G)$ satisfy  $|V_i| \ge \Delta(H)+2$ for all $i$, then there exists an independent set in $G$ (namely a matching in $H$) meeting all $V_i$'s.

\begin{remark}

\begin{enumerate}
\item  We know only one example, taken from \cite{jin,yuster}, in which $|V_i| \ge \Delta(H)+1$ does not suffice.
Take three vertex disjoint copies of $C_4$, say $A_1,A_2,A_3$. Number
the edges of $A_i$ cyclically as $a_i^j ~(j=1\ldots 4)$.
Let $E_1=\{a_1^1,a_1^3,a_3^1\}$, ~~ $E_2=\{a_1^2,a_1^4,a_3^3\}$,
$E_3=\{a_2^1,a_2^3,a_3^2\}$ and $E_4=\{a_2^2,a_2^4,a_3^4\}$.
Then $\Delta(\bigcup_{i \le m} E_i)=2$, ~~$|E_i|=3$ and there is no
rainbow matching.

\item The conjecture is false if the sets $E_i$ are allowed to be multisets. We omit the example showing this.
\end{enumerate}

\end{remark}

An even stronger version of the conjecture is:
\begin{conjecture}
If the edge set of a graph $H$ is partitioned into sets $E_1, \ldots ,E_m$ then there exists a matching $M$ satisfying $|M \cap E_i| \ge \floor*{ \frac{|E_i|}{\Delta(H)+2}}$  for all $i \le m$

\end{conjecture}

\subsection{Over-representation vs. under-representation and representing general systems of sets}

It is easy to find examples falsifying the above conjectures when the sets that are to be fairly represented do not form a partition. Why is that? A possible explanation is that  a more natural formulation of our conjectures is not in terms of over-representation, but in terms of under-representation by a large set. Here is a   conjecture in this direction:

\begin{conjecture}
For every $m$ there exists a number $c(m)$ for which the following is true: if $G$ is a bipartite graph and  $E_1, \ldots ,E_m$ are any sets of edges, then there exists a matching $S$ in $G$ of size at least $\frac{|E(G)|}{\Delta(G)}-c(m)$ such that

\begin{equation} \label{underrep}
|S \cap E_i|\le \lceil\frac{|E_i|}{\Delta(G)}\rceil~~\text{ for all }~~i \le m
\end{equation}
\end{conjecture}

Possibly $c(m)=\frac{m}{2}$ may suffice. When the $E_i$'s form a partition,
condition \eqref{underrep} implies that all but $c(m)$ sets are fairly represented.
Of course, a stronger condition is required to imply Conjecture \ref{equirep00}.
The reason that the under-representation formulation is  natural is that if the sets $E_i$ form a partition, the condition in \eqref{underrep} defines a generalized partition matroid.  The conjecture thus concerns  representation by a set belonging to the intersection of three matroids.

\section{Fair representation by independent sets in paths:  a Borsuk-Ulam  approach}\label{paths}

In this section we  prove Theorem \ref{pathscasetotal}. Following an idea from the proof of the ``necklace theorem" \cite{Al}, we shall use the Borsuk-Ulam theorem.
In the necklace problem two thieves want to divide a necklace with $m$ types of beads, each occurring in an even number of beads, so that
the beads of every type are evenly split between the two.
 The theorem is that the thieves can achieve
this goal using at most $m$ cuts of the necklace. In our case, we shall employ as ``thieves'' the sets of odd and even points, respectively, in a sense to be explained below.

We first quote the Borsuk-Ulam theorem. As usual, for $n\ge 1$, $S^{n}$ denotes the
set of points $\vec{x}=(x_1, \ldots, x_{n+1}) \in \mathds{R}^{n+1}$ satisfying $\sum_{i \le n+1}x_i^2=1$.

\begin{theorem}[Borsuk-Ulam]
For all $n\ge 1$, if $f: S^{n}\rightarrow \mathds{R}^{n}$ is a continuous odd function, then there exists $\vec{x}\in S^n$ such that $f(\vec{x})=0$.
\end{theorem}

\begin{proof}[Proof of Theorem \ref{pathscasetotal}]
Let  $v_1, \ldots ,v_n$ be the vertices of $P_n$, ordered along the path. In order to use the Borsuk-Ulam theorem, we first make the problem continuous, by replacing each vertex $v_p$ by the characteristic function of the $p$th of $n$ intervals of length $\frac{1}{n}$ in $[0,1]$, open on the left and closed on the right, except for the firs interval which is closed on both sdeis.  We call
the interval $(\frac{p-1}{n},\frac{p}{n}]$ ($[0,\frac{1}{n}]$ for $p=1$) a {\em bead} and denote it by $B_p$. Let $\chi_i$ be the characteristic function of
$\bigcup_{v_p \in V_i}B_p$.  Let $g$ be the characteristic function of the union of odd beads on the path, and let $h(y)=1-g(y)$.

Given a point $\vec{x}\in S^m$, let $z_k = \sum_{j\le k} x_j^2$, for all $k=0,\ldots,m+1$ (where $z_0=0, z_{m+1}=1$).

For each $i \le m$ define a function $f_i: ~S^{m} \to \mathds{R}$ by:

$$f_i(x_1, \ldots, x_m)=\sum_{1 \le k \le m}\int_{z_{k-1}}^{z_k}(g(y) - h(y))\chi_i(y) \sign(x_k)dy$$

Here, as usual, $\sign(x)=0$ if $x = 0$,~~ $\sign(x)=1$ if $x > 0$,~~ and $\sign(x)=-1$ if $x<0$.
Since  the set of points of discontinuity of the sign function is discrete, the functions $f_i$ are continuous.
The sign term  guarantees that $f_i(-\vec{x})=-f_i(\vec{x})$. Hence, by the Borsuk-Ulam theorem there exists a point
$\vec{w}=(w_1, \ldots,w_{m+1}) \in S^{m}$ such that $f_i(\vec{w})=0$ for all $i \in [m]$, where $z_k=\sum_{j\le k} w_j^2$, for all $k=0,\ldots,m+1$.

For $y \in [0,1]$ such that $y \in (z_{k-1},z_k]$ define $POS(y)=1$ if $w_k\ge 0$ and $POS(y)=0$ otherwise. Let $NEG(y)=1-POS(y)$.
Let

$$J_1(y)=POS(y)g(y)+ NEG(y)h(y), ~~J_2(y)=POS(y)h(y)+ NEG(y)g(y).$$

For fixed $i \in [m]$, the fact that $f_i(\vec{w})=0$ means that


$$\int_{y=0}^1 \chi_i(y)POS(y) [g(y)-h(y)]dy   = \int_{y=0}^1 \chi_i(y) NEG(y)[g(y)-h(y)]dy$$

Shuffling terms this gives:

\begin{equation}\label{balanced1} \int_{y=0}^1 \chi_i(y)[POS(y) g(y)+ NEG(y)h(y)]dy   = \int_{y=0}^1 \chi_i(y)[POS(y) h(y)+ NEG(y)g(y)]dy
\end{equation}

Denoting the integral $\int_0^1u(y)dy$ of a function $u$ by $|u|$, and noting that $J_1(y)+J_2(y)=1$ for all $y\in[0,1]$,
Equation \eqref{balanced1} says that
\begin{equation}\label{balanced} |\chi_iJ_1| = |\chi_iJ_2|=\frac{|\chi_i|}{2} \end{equation}

for every $i \le m$.

    A bead contained in an interval  $(z_{k-1},z_k]$ is called {\em positive} if $w_k\ge 0$ and {\em negative} otherwise.
    For $k=1, \ldots ,m$ let $T_k$ be the bead containing $z_k$. The beads that are equal to $T_k$ for some $k$ are called  {\em transition beads}. Let $F$ be the set of transition beads, and let $Z= \bigcup F$. We next remove the transition beads from $J_j$, by defining:

$$\tilde{J}_j(y)=\min(J_j(y),1-\chi_Z(y))$$

Thus $\tilde{J}_1$ is the characteristic function of the union of those beads that are either positive and odd, or negative and even, and $\tilde{J}_2$ is the characteristic function of the union of those beads that are either positive and even, or negative and odd. Let $I_j ~~(j=1,2)$ be the set of vertices $v_p$  on whose bead $B_p$ the function $\tilde{J}_j$ is positive.
Since the transition beads have been removed, $I_1$ and $I_2$ are independent.

For  $i \le m$ and $j=1,2$ let $c(i,j)$ be the amount of loss of $\tilde{J_j}$ with respect to $J_j$ on beads
belonging to $V_i$, namely beads $B_p \in F$ such that $v_p \in V_i$. Formally,

$$c(i,j)
=\sum_{v_p \in V_i, B_p \in F}|\chi_{B_p}\cdot(J_j-\tilde{J}_j)|$$

Then
\begin{equation}\label{lossi} c(i,1)+c(i,2)=\frac{1}{n}|\{v_p \in V_i, B_p \in F\}|
\end{equation}
 and $\sum_{i \le m} c(i,1)+\sum_{i \le m} c(i,2)=\frac{m}{n}$. Hence
for either $j=1$ or $j=2$ we have
$\sum_{i \le m} c(i,j)\le \frac{m}{2n}$. Let $I=I_j$ for this particular $j$, and denote $c(i,j)$ by $b_i$.
Then, by \eqref{balanced} and \eqref{lossi} we have $|I \cap V_i| \ge \frac{|V_i|}{2}-b_i$, while $\sum_{i \le m}b_i \le \frac{m}{2}$. Namely, the set $I$ satisfies the conditions of the theorem.

\end{proof}

\begin{example} Let $P=P_4$, the path with $4$ vertices $v_i, 1 \le i \le 4$, and let $V_1=\{v_1,v_2,v_4\}$ and $V_2=\{v_3\}$. Then one possible set of points  given by the Borsuk-Ulam theorem is $z_1=\frac{1}{8},~z_2=\frac{5}{8}$, and $w_1>0,~w_2<0,~w_3 >0$ (or with all three signs reversed), as illustrated in Figure~\ref{fig0}. Thus, $J_1$ is the characteristic function of $[0,\frac{1}{8}]\cup (1,2] \cup(\frac{5}{8},\frac{3}{4}]$ and $J_2$ is the characteristic function of $(\frac{1}{8},1] \cup  (\frac{1}{2},\frac{5}{8}]\cup (\frac{3}{4},1]$. The set $I_1$ is obtained from $J_1$ by removing the $z_i$-infected beads, namely $I_1=\{v_2\}$, and then $I_2=\{v_4\}$. In this case $\sum_{i \le m} c(i,j)= \frac{m}{2n}=\frac{2}{4}=\frac{1}{2}$
for both $j=1$ and $j=2$, and thus we can choose $I$ as either $I_1$ or $I_2$. This is what the proof gives, but in fact in this example we can do better - we can take $I=\{v_1,v_3\}$, in which only $V_1$ is under-represented.

\begin{figure}[h!]
\begin{center}
\includegraphics[scale=0.3]{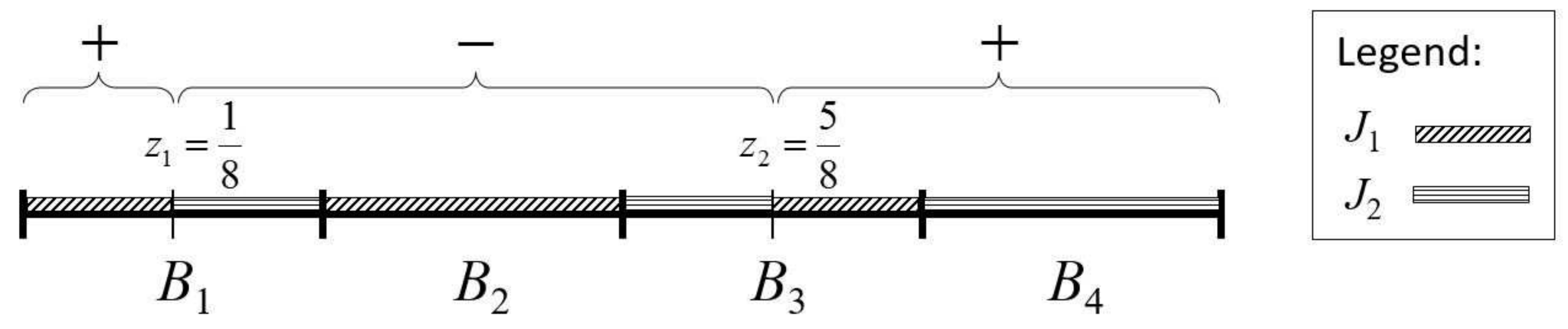}
\end{center}
\caption{An example with four vertices divided into two sets.}
\label{fig0}
\end{figure}

\end{example}

\begin{remark} \nopagebreak \hfill
\begin{enumerate}
\item
The inequality $\sum_{i \le m}b_i \le \frac{m}{2}$ can possibly be
improved, but not much.
Namely, there are examples in which  the minimum of the sum $\sum_{i
\le m}b_i$ in the theorem is $\frac{m-1}{2}$. To see this, let $m=2k+1$, and let each $V_i$ be of size $2k$. Consider a sequence of length $2k \times (2k+1)$, in which the $(i-1)m+2j-1$-th element belongs to $V_i$ ($i =1, \ldots, 2k, ~~j=1, \ldots,k+1$), and the rest of the elements are chosen in any way so as to satisfy the condition $|V_i|=2k$. For example, if $k=2$ then the sequence is of the form:

$$1*1*1-2*2*2-3*3*3-4*4*4$$

where the $*$s can be filled in any way that satisfies $|V_i|=4$ (namely, four of them are replaced by the symbol $5$, and one is replaced by $i$ for each symbol $i=1,2,3,4$. The dashes are there to facilitate the reference to the four stretches). If $S$ is an independent set in the path then we may assume that $S$ contains no more than $k$ elements from the same $V_i$ from each stretch (for example, in the first stretch of the example above  choosing all three $1$s will result in deficit of 2 in the other sets),
 Thus $|S| \le 2k\times k$, which is  $\frac{m-1}{2}$ short of half the length of the path.

\item It may be of interest to find the best bounds as a function of the sizes of the sets $V_i$ and their number. Note that in the example above the size of the sets is almost equal to their number. As one example, if
 all $V_i$'s are of size $2$, then the inequality can be improved to: $\sum_{i \le m}b_i \le \frac{m}{3}$. To see this, look at the multigraph obtained by adding to $P_n$ the pairs forming the sets $V_i$ as edges.
In the resulting graph the maximum degree is $3$, and hence by Brooks' theorem it is $3$-colorable. Thus there is an independent set of size at least $\frac{n}{3}$, which represents all $V_i$'s apart from at most $\frac{m}{3}$ of them.

\end{enumerate}

\end{remark}

\section{Fair representation by independent sets in cycles: using a theorem of Schrijver}\label{paths2}

In this section we shall prove Theorem \ref{pathscaseindividual}.  The proof
uses a result of Schrijver \cite{sc}, which is a strengthening of a theorem of Lov\'asz:

\begin{theorem}[Schrijver \cite{sc}]
\label{t21}
For integers $k,n$ satisfying $n>2k$ let $K=K(n,k)$ denote the graph whose vertices
are all  independent sets of size $k$ in a cycle $C$ of length
$n$, where two such vertices are adjacent iff the corresponding
sets are disjoint. Then the chromatic number of $K$ is
$n-2k+2$.
\end{theorem}

The hard part of this inequality is that the chromatic number of $K$ is at least
$n-2k+2$, which can be formulated as follows:

\begin{theorem}
\label{t211}
The family $\ci(n,k)$ of independent sets of size $k$ in the cycle $C_{n}$ cannot be partitioned into fewer than  $ n-2k+2$ intersecting families.
\end{theorem}

We start with a simple case, in which all $V_i$'s but one are odd:
\begin{theorem}
\label{t11}
Let $m,r_1,r_2,\ldots ,r_m$ be positive integers, and put
$n=\sum_{i=1}^m (2r_i+1)-1$.
Let $G=(V,E)$ be a cycle of length $n$, and let
$V=V_1 \cup V_2 \cup \ldots \cup V_m$ be a partition of its vertex set, where
$|V_i|=2r_i+1$ for all $1 \leq i <m$ and $|V_m|=2r_m$.
Then there is an independent set $S$ of $G$ satisfying
$|S|=\sum_{i=1}^m r_i $ and $|S \cap V_i|=r_i$ for all $1 \leq i \leq m$.
\end{theorem}
\vspace{0.2cm}

\noindent
{\bf Proof of Theorem \ref{t11}:}\,
Put $k=\sum_{i=1}^m r_i$
and note that $n-2k+2=m+1>m$. Assume, for contradiction,
that there is
no $S \in \ci(n,k)$ satisfying
the assertion of the theorem. Then  for every $S \in \ci(n,k)$
there is at least one index $i$ for which $|S \cap V_i| \geq r_i+1$.
Indeed, otherwise
$|S \cap V_i| \leq r_i$ for all $i$ and hence $|S \cap V_i|=r_i$ for
all $i$, contradicting the assumption.
Let $\cf_i$ be the family of sets $S \in \ci(n,k)$ for which $|S \cap V_i| \geq r_i+1$. Clearly, $\cf_i$ is intersecting (in fact, intersecting within $V_i$), contradicting
the conclusion of Theorem \ref{t211}.

 \hfill $\Box$
\vspace{0.2cm}

\begin{corollary}
\label{c12}
Let   $V=V_1 \cup V_2 \cup \cdots \cup V_m$ be a partition of
the vertex set of a cycle $C$.
\vspace{0.1cm}

\noindent
(i) For every $i$ such that  $|V_i|$ is even there exists an
independent set $S_i$ of $C$ satisfying:

\begin{enumerate}
\item $|S_i \cap V_i|=|V_i|/2$.
\item  $|S_i \cap V_j|=(|V_j|-1)/2$ for all $j$ for which $|V_j|$ is odd.
\item $|S \cap V_j| =|V_j|/2 -1$ for every $j \neq i$ for which $|V_j|$ is even.
.
\end{enumerate}
\vspace{0.1cm}

\noindent
(ii) If~ $|V_i|$ is odd for all $i\le m$ then for any vertex $v$
of $C$ there is an
independent set $S$ of $C$ not containing $v$ and satisfying
$|S \cap V_i|=(|V_i|-1)/2$ for all $i$.
\end{corollary}

\noindent
{\bf Proof of Corollary \ref{c12}:}\,
Part (i) in case all sets $V_j$ besides $V_i$ are of odd sizes
is exactly the assertion of Theorem \ref{t11}. If there are
additional indices $j \neq i$ for which  $|V_j|$ is even, choose
an arbitrary vertex from each of them and contract an edge incident
with it. The result follows by applying the theorem
to the shorter cycle obtained. Part (ii) is proved in the same way,
contracting an edge incident with $v$. \hfill $\Box$


\section{More applications of Schrijver's theorem and its extensions }

\subsection{Hypergraph versions}
The results above can be extended by applying known hypergraph
variants of Theorem \ref{t21}. For integers $n \geq s \geq 2$, let
$C_n^{s-1}$ denote the $(s-1)$-th power of a cycle of length $n$,
that is, the graph obtained from a cycle of length $n$ by
connecting every two vertices whose distance in the cycle is at
most $s-1$. Thus if $s=2$ this is simply the cycle of length $n$
whereas if $n \leq 2s-1$ this is a complete graph on $n$ vertices.
For integers $n,k,s$ satisfying $n>ks$, let $K(n,k,s)$ denote the
following $s$-uniform hypergraph. The vertices are all independent
sets of size $k$ in $C_n^{s-1}$, and a collection $V_1,V_2, \ldots, V_s$
of such vertices forms an edge iff the sets $V_i$ are pairwise
disjoint. Note that for $s=2$, $K(n,k,2)$ is exactly the graph
$K(n,k)$ considered in Theorem \ref{t21}.   The following
conjecture appears in \cite{ADL}.
\begin{conjecture}
\label{c31}
For $n >ks$, the chromatic number of $K(n,k,s)$ is
$\lceil \frac{n-ks+s}{s-1} \rceil$.
\end{conjecture}
This is proved in \cite{ADL} if $s$ is any power of $2$.
Using this fact we can prove the following.
\begin{theorem}
\label{t32}
Let $s \geq 2$ be a power of $2$, let
$m$ and $r_1,r_2, \ldots ,r_m$ be integers, and
put $n=s\sum_{i=1}^m r_i +(s-1)(m-1)$. Let
$V_1,V_2, \ldots , V_m$ be a partition of the vertex set of
$C_n^{s-1}$, where
$|V_i|=sr_i+s-1$ for all $1 \leq i <m$, and $|V_m|=sr_m$.
Then there exists an independent set $S$ in $C_n^{s-1}$ satisfying
$|S \cap V_i|=r_i$ for all $1 \leq i \leq m$.
\end{theorem}
{\bf Proof:}\,
Put $k=\sum_{i=1}^m r_i$
and note that the chromatic number of
$K(n,k,s)$ is $\lceil (n-ks+s)/(s-1) \rceil>m$. Assume, for contradiction,
that there is a partition of the vertex set of
$C_n^{s-1}$ with parts $V_i$
as in the theorem, with
no independent set of $C_n^{s-1}$ of size $k=\sum_{i=1}^m r_i$ satisfying
the assertion of the theorem. In this case, for any such independent
set $S$
there is at least one index $i$ so
that $|S \cap V_i| \geq r_i+1$.
We can thus define a
coloring $f$ of the independent sets of size $k$ of $C_n^{s-1}$ by letting
$f(S)$ be the smallest $i$ such that $|S \cap V_i| \geq r_i+1$.
Since the chromatic number of $K(n,k,s)$ exceeds $m$, there are $s$
pairwise
disjoint sets $S_1,S_2, \ldots ,S_s$
and an index $i$ such that $|S_j \cap V_i|
\geq r_i+1$ for all $1 \leq j \leq s$. But this implies that
$|V_i| \geq sr_i+s$, contradicting the assumption on the size of the
set $V_i$, and completing the proof. \hfill $\Box$
\vspace{0.2cm}

\noindent
Just as in the previous section, this implies the following.
\begin{corollary}
\label{c33}
Let $s>1$ be a power of $2$.
Let $V_1, V_2, \ldots ,V_m$ be a partition of the vertex set of
$C_n^{s-1}$, where $n=\sum_{i=1}^m |V_i|$. Then there is an
independent set  $S$ in $C_n^{s-1}$ satisfying
$$
|S \cap V_i| = \floor*{ \frac{|V_i|-s+1}{s} }
$$
for all $1 \leq i <m$, and
$$
|S \cap V_m| = \floor*{ \frac{|V_i|}{s} }.
$$
\end{corollary}
The proof is by contracting edges, reducing each set $V_i$ to one
of size $s \floor*{ \frac{|V_i|-s+1}{s} } +s-1$ for $1 \leq i <m$,
and reducing $V_m$ to a set of size $s \floor*{ \frac{|V_m|}{s}
}$.
The result follows by applying Theorem \ref{t32} to this contracted
graph.
\vspace{0.2cm}

\subsection{The Du-Hsu-Wang conjecture}
Du, Hsu and Wang \cite{DHH} conjectured that
if a graph on $3n$ vertices is the
edge disjoint union of a Hamilton cycle of length $3n$ and $n$ vertex
disjoint
triangles  then its independence number is $n$.  Erd\H{o}s
conjectured that in fact any such graph is
$3$-colorable. Using an algebraic approach introduced in \cite{AT}, Fleischner and Stiebitz \cite{FS}
proved this conjecture in a stronger form - any such graph is in fact
$3$-choosable.

The original conjecture, in a slightly stronger form, can be derived from Theorem \ref{t11}:
omit any  vertex and apply the theorem with $r_i=1$ for all $i$. So, for every vertex $v$ there exists a representing set as desired in the conjecture  omitting $v$.
The derivation of the statement of Theorem
\ref{t11} from the result of Schrijver in \cite{sc} actually
supplies a quick proof of the following:

\begin{theorem}
Let $C_{3n}=(V,E)$ be  cycle of length $3n$ and let
$V= A_1 \cup A_2 \cup  \ldots  \cup A_n$ be a partition of its
vertex set into $n$ pairwise disjoint sets, each of size $3$.
Then there exist two disjoint independent
sets in the cycle, each containing one point from each $A_i$.
\end{theorem}

\begin{proof}

Define a coloring of the independent sets of size $n$ in
$C_{3n}$ as follows. If $S$ is such an independent set and
there is an index $i$ so that $|S \cap A_i| \geq 2$, color $S$ by
the smallest such $i$. Otherwise, color $S$ by the color $n+1$. By
\cite{sc} there are two disjoint independent sets $S_1,S_2$ with
the same color. This color cannot be any $i \leq n$, since if this
is the case then
$$
|(S_1 \cup S_2) \cap A_i| =|S_1 \cap A_i|+|S_2 \cap A_i| \geq
2+2=4>3=|A_i|,
$$
which is impossible. Thus $S_1$ and $S_2$ are both colored $n+1$,
meaning that each of them contains exactly one element of each
$A_i$. \end{proof}

The Fleischner-Stiebitz theorem implies that the representing set in the DHW conjecture can be required to contain any given vertex.
 This
 can also be deduced  from the topological version of Hall's
Theorem first proved in \cite{ah} (for this derivation see e.g \cite{ahhs}). The latter shows also that the cycle of length $3n$ can be replaced by a union of cycles, totalling $3n$ vertices,  none being of length $1 \bmod 3$.
Simple examples show that the Fleischner-Stiebitz theorem on $3$-colorability does not apply to this setting.

Note that none of the above proofs supplies an
efficient algorithm for finding the desired  independent set.

\section{Fair representation by matchings in $K_{n,n}$, the case of two parts}

The case $m=2$ of Conjecture \ref{equirep00} is easy. Here is its statement in this case:

\begin{theorem}\label{theorem:m=2}
If $F$ is a subset of $E(K_{n,n})$,  then
there exists a perfect matching $N$ such that $|N \cap F|\ge \floor*{ \frac{|F|}{n}}-1$ and $|N \setminus F|\ge \floor*{ \frac{|E(G)\setminus F|}{n}}-1$.
\end{theorem}

Partitioning $E(K_{n,n})$ into $n$ perfect matchings shows that there exist two perfect matchings, $N_1$ and $N_2$, such that $|N_1 \cap F| \le \frac{|F|}{n} \le |N_2 \cap F|$. The fact that any permutation can be reached from any other by a sequence of transpositions means that it is possible to reach $N_2$ from $N_1$ by a sequence of exchanges,  replacing at each step two edges of the perfect matching by two other edges.
Thus, by a mean value argument, at some matching in the process the condition is satisfied.

   The question remains of determining the cases in which the  $(-1)$ term is necessary. That this term is sometimes necessary is shown, for example, by the case of $n=2$ and $F$ being a perfect matching. Another example - $n=6$ and
$F=\left ( [3]\times [3]\right ) \cup \left (\{4,5,6\} \times \{4,5,6\}\right )$: it is easy to see that there is no perfect matching containing precisely  $3$ edges from $F$, as required in Conjecture \ref{equirep00}.

The appropriate condition  is given by the following concept:

\begin{definition}\label{def:rigidity}
A subset $F$ of $\ekn$ is said to be {\em rigid} if there exist subsets $K$ and $L$ of $[n]$ such that
$F=K \times L \cup ([n]\setminus K)\times ([n]\setminus L)$.
\end{definition}

The  rigidity in question is with respect to $F$-parity  of perfect matchings:

\begin{theorem}\cite{amw}
A subset $F$ of $\ekn$ is rigid if and only if $|P \cap F|$ has the same   parity for all perfect matchings $P$ in $\kn$.
\end{theorem}

This characterization  shows that when $F$ is rigid, it is not always
possible to drop the ``minus $1$'' term in Theorem \ref{theorem:m=2}. Conversely, if $F$ is not rigid, then the ``minus $1$'' term can indeed be dropped, as indicated by Corollary~\ref{theorem:m=2nonrigidcase} below.

We shall show:

\begin{theorem}\label{theorem:d}
Let $a<c<b$ be three integers and suppose that $F \subseteq E(K_{n,n})$ is not rigid.
If there exists a perfect matching $P_a$ such that $|P_a\cap F|=a$ and a perfect matching $P_b$ such that $|P_b\cap F|=b$,   then there exists a perfect matching $P_c$ satisfying $|P_c\cap F|=c$.
\end{theorem}

It follows from Theorem~\ref{theorem:d} that if a subset $F$ of $E(K_{n,n})$ is not rigid
then for every integer $c$ such that $n - \nu(E(K_{n,n}) \setminus F) \le c \le \nu(F)$
there exists a perfect matching $N$ satisfying  $|N\cap F|=c$. This implies,

\begin{corollary}\label{theorem:m=2nonrigidcase}
If  a subset $F$  of $E(K_{n,n})$
is not rigid, or if $ n \nmid  |F|$, then there exists a perfect matching $N$ such that $|N \cap F|\ge \floor*{ \frac{|F|}{n}}$ and $|N \setminus F|\ge \floor*{ \frac{|E(K_{n,n})\setminus F|}{n}}$.
\end{corollary}

\begin{proof}[Proof of Theorem~\ref{theorem:d}]
We  use  the matrix language of the original Ryser conjecture (Section~\ref{sec:known}).
Let $M$ be the $n \times n$ matrix in which $m_{i,j}=1$ if $(i,j)\in F$ and $m_{i,j}=0$ if $(i,j)\not \in F$. A perfect matching in $G$ corresponds to a \emph{generalized diagonal} (abbreviated \emph{GD}) in $M$, namely
a set of $n$ entries belonging to distinct rows and columns. A GD will be called a $k$-\emph{GD} if exactly $k$ of its entries are 1.
 By assumption there exist an $a$-GD $T^a$ and a $b$-GD $T^b$. Assume, for contradiction, that there is no $c$-GD.
The case $n=2$ is trivial, and hence, reversing the roles of $0$s and $1$s if necessary, we may assume that $c>1$.
Since a GD corresponds to a permutation in $S_n$, and since every permutation can be obtained from any other permutation by a sequence of transpositions,
there exists a sequence of GD's $T^a=T_1, T_2, \ldots, T_k=T^b$, where each pair $T_i$ and $T_{i+1}$, $i=1,\ldots,k-1$, differ in two entries.
By the contradictory assumption there exists $i$ such that $T:=T_{i+1}$ is a $(c+1)$-GD and $T':=T_i$  is a $(c-1)$-GD. Without loss of generality
we may assume that $T$ lies along the main diagonal, its first $c+1$ entries are 1, and the rest of its entries are 0.

Let $I=[c+1],~J=[n]\setminus I$ and let $A=M[I \mid I], ~B=M[I \mid J], ~C=M[J \mid I], ~D=M[J \mid J]$ (we are using here a common notation - $M[I \mid J]$ denotes the submatrix of $M$ induced by the row set $I$ and column set $J$).
 We may  assume that the GD $T'$ is obtained from $T$ by replacing
the entries $(c,c)$ and $(c+1,c+1)$ by $(c+1,c)$ and $(c,c+1)$ (Figure~\ref{fig1}).

\begin{figure}[h!]
\begin{center}
\includegraphics[scale=0.2]{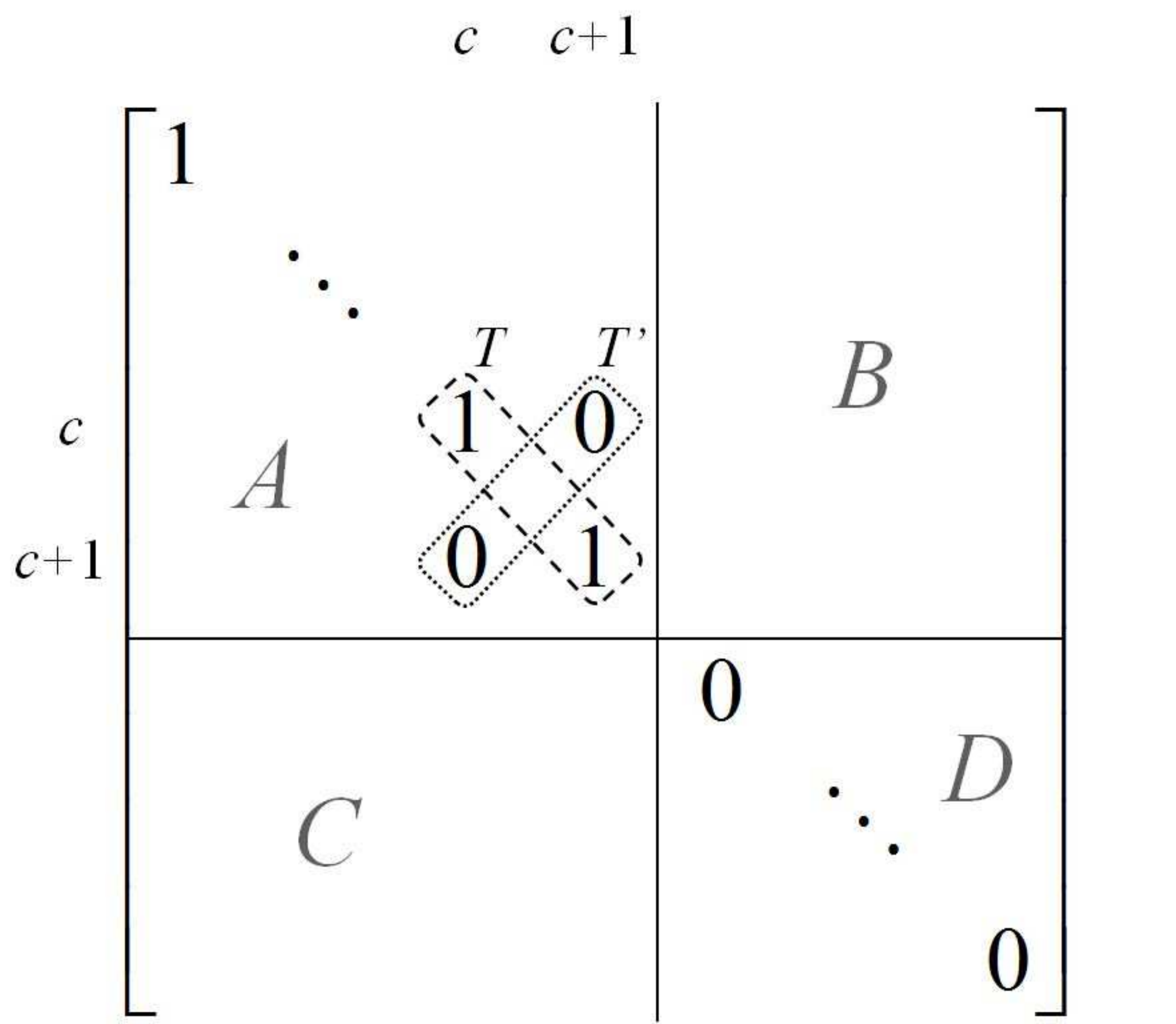}
\end{center}
\caption{}
\label{fig1}
\end{figure}

\begin{claim}\label{claim:1}
The matrices $A$ and $D$ are symmetric.
\end{claim}
\begin{proof}[Proof of Claim~\ref{claim:1}]
\renewcommand{\qedsymbol}{}
To prove that $A$ is symmetric, assume, for contradiction, that there exist $i_1 \neq i_2 \in I$ such that
 $m_{i_1,i_2}\neq m_{i_2,i_1}$. Then, we can replace the entries $(i_1,i_1)$ and $(i_2,i_2)$ in $T$ by $(i_1,i_2)$ and $(i_2,i_1)$ to obtain a $c$-GD. The proof for $D$ is similar,  applying  the replacement in this case to $T'$.
\end{proof} 

\begin{claim}\label{claim:2}
If $i \in I$ and $j \in J$ then
 $m_{i,j}\ne m_{j,i}$.
\end{claim}

\begin{proof}[Proof of Claim~\ref{claim:2}]
\renewcommand{\qedsymbol}{}

{\bf Case $I$:} $m_{i,j}=m_{j,i}=0$. Replacing $(i,i)$ and $(j,j)$ in $T$ by $(i,j)$ and $(j,i)$ results in a $c$-GD.

{\bf Case $II$:}
 $m_{i,j}=m_{j,i}=1$.

 {\bf Subcase $II_1$:} $i \not \in \{c,c+1\}$. Replacing in $T'$ the entries $(i,i)$ and $(j,j)$ by $(i,j)$ and $(j,i)$ results in a $c$-GD.

  {\bf Subcase $II_2$:} $i  \in \{c,c+1\}$.
 Without loss of generality we may assume $i=c+1$ and $j=c+2$ (Figure~\ref{fig2}). If $m_{k,\ell}=m_{\ell,k}=0$ for some $1\le k<\ell\le c$  then  replacing in $T$ the entries $(k,k),(\ell,\ell),(c+1,c+1)$ and $(c+2,c+2)$  by $(k,\ell),(\ell,k),(c+1,c+2)$ and $(c+2,c+1)$ results in a $c$-GD  (Figure~\ref{fig2}). Thus, we may assume that $m_{k,\ell}=m_{\ell,k}=1$  for all $k,\ell \le c$.

\begin{figure}[h!]
\begin{center}
\includegraphics[scale=0.25]{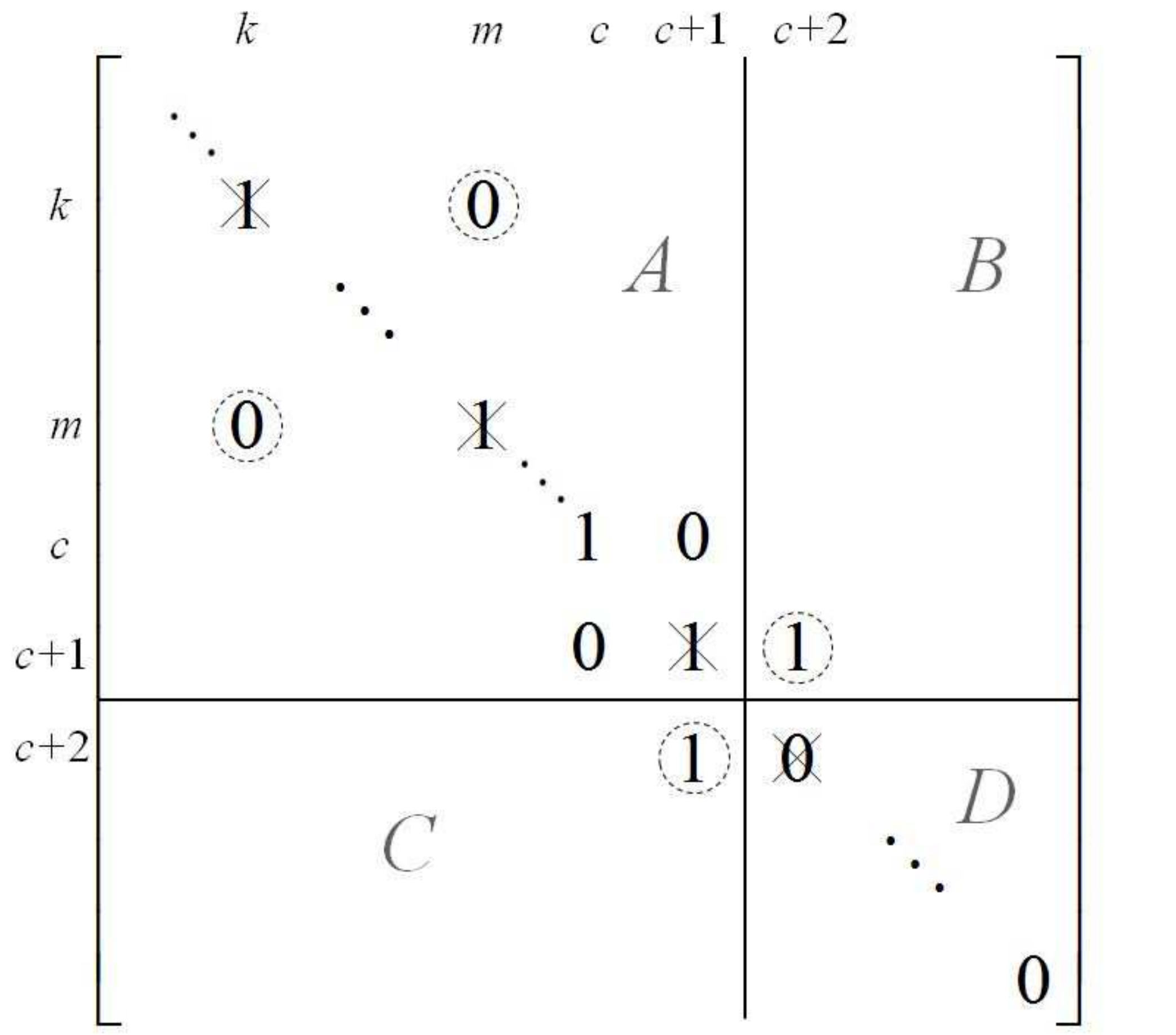}
\end{center}
\caption{Subcase $II_2$. Removed entries are struck out by $\times$ and added entries are circled.}
\label{fig2}
\end{figure}

We now consider three sub-subcases:

(i) $m_{c,c+2}=0, m_{c+2,c}=1$. In this case we may replace the entries $(c,c),(c+1,c+1)$ and $(c+2,c+2)$ in $T$ by $(c,c+2),(c+1,c)$ and $(c+2,c+1)$ and obtain a $c$-GD (Figure~\ref{fig3a}).

(ii) $m_{c,c+2}=1, m_{c+2,c}=0$. Replace the same entries as in Case (i) by $(c,c+1),(c+1,c+2)$ and $(c+2,c)$, again obtaining a $c$-GD (Figure~\ref{fig3b}).

(iii) $m_{c,c+2}=m_{c+2,c}=1$. If $m_{c-1,c+1}=0$ then, remembering that $m(c-1,c-1)=1$, we can replace $(c-1,c-1),(c,c),(c+1,c+1)$ and $(c+2,c+2)$ in $T$ by $(c-1,c+1),(c,c)+2),(c+1,c-1),(c+2,c)$ and obtain a $c$-GD (Figure~\ref{fig4a}). If $m_{c-1,c+1}=1$, we can replace $(c-1,c-1),(c,c)$ and $(c+1,c+1)$ in $T$ by $(c-1,c+1),(c,c-1)$ and $(c+1,c)$ and obtain a $c$-GD (Figure~\ref{fig4b}.). This proves Claim~\ref{claim:2}.

\begin{figure}[h!]
  \centering
  \subfigure[]{\label{fig3a}\includegraphics[scale=0.3]{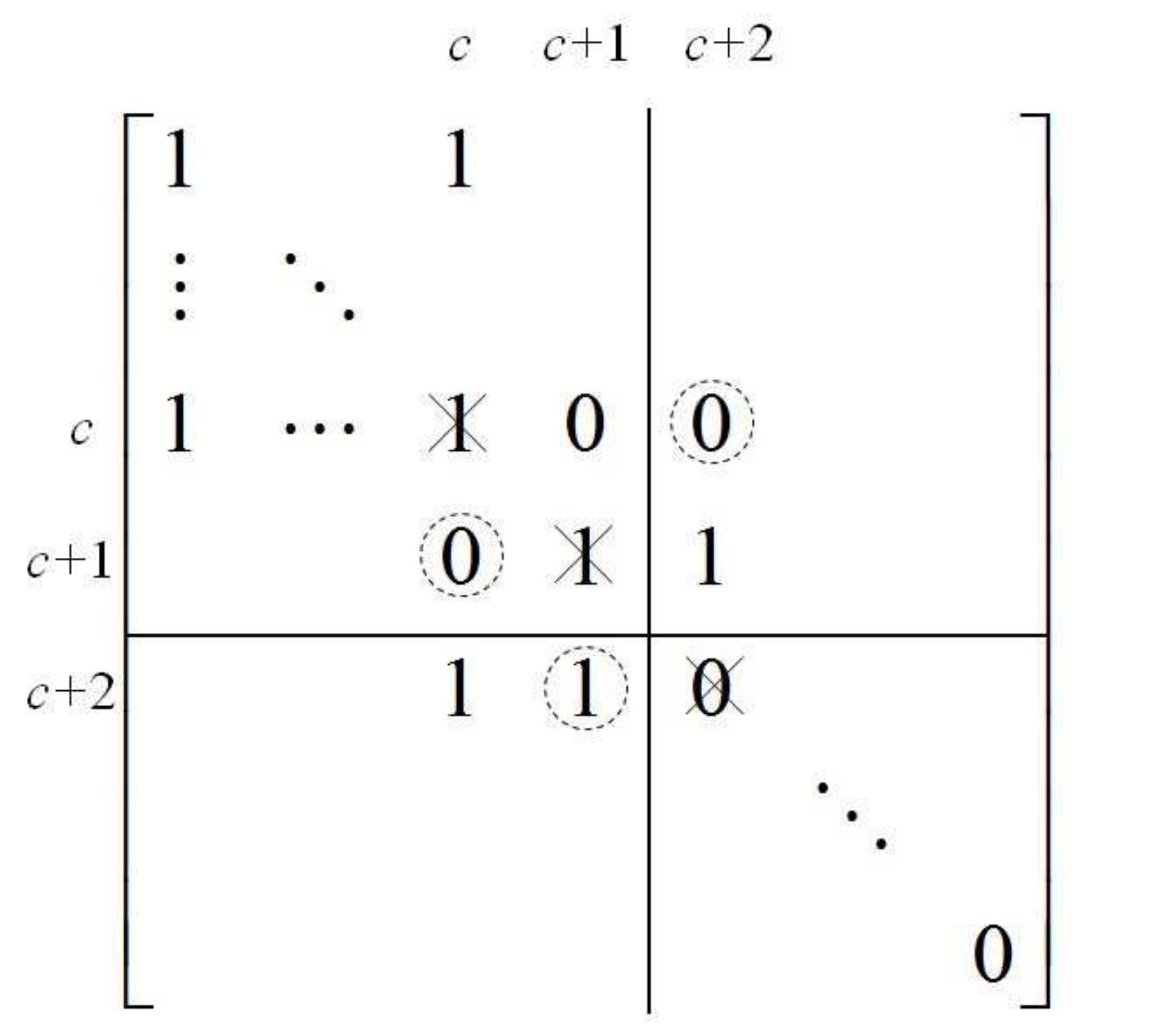}}
  \subfigure[]{\label{fig3b}\includegraphics[scale=0.3]{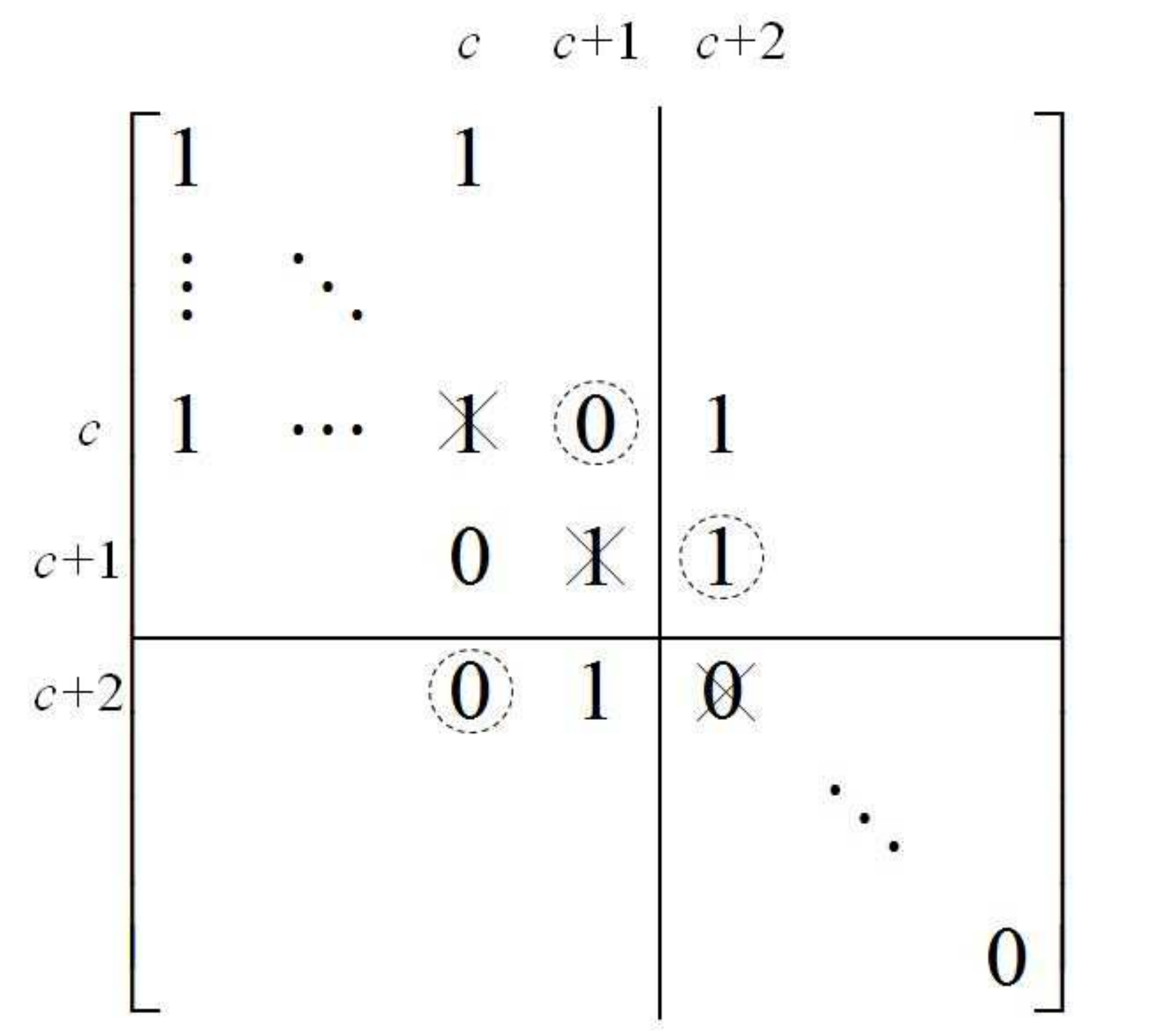}}
  \caption{}
  \label{fig3}
\end{figure}

\begin{figure}[h!]
  \centering
  \subfigure[]{\label{fig4a}\includegraphics[scale=0.27]{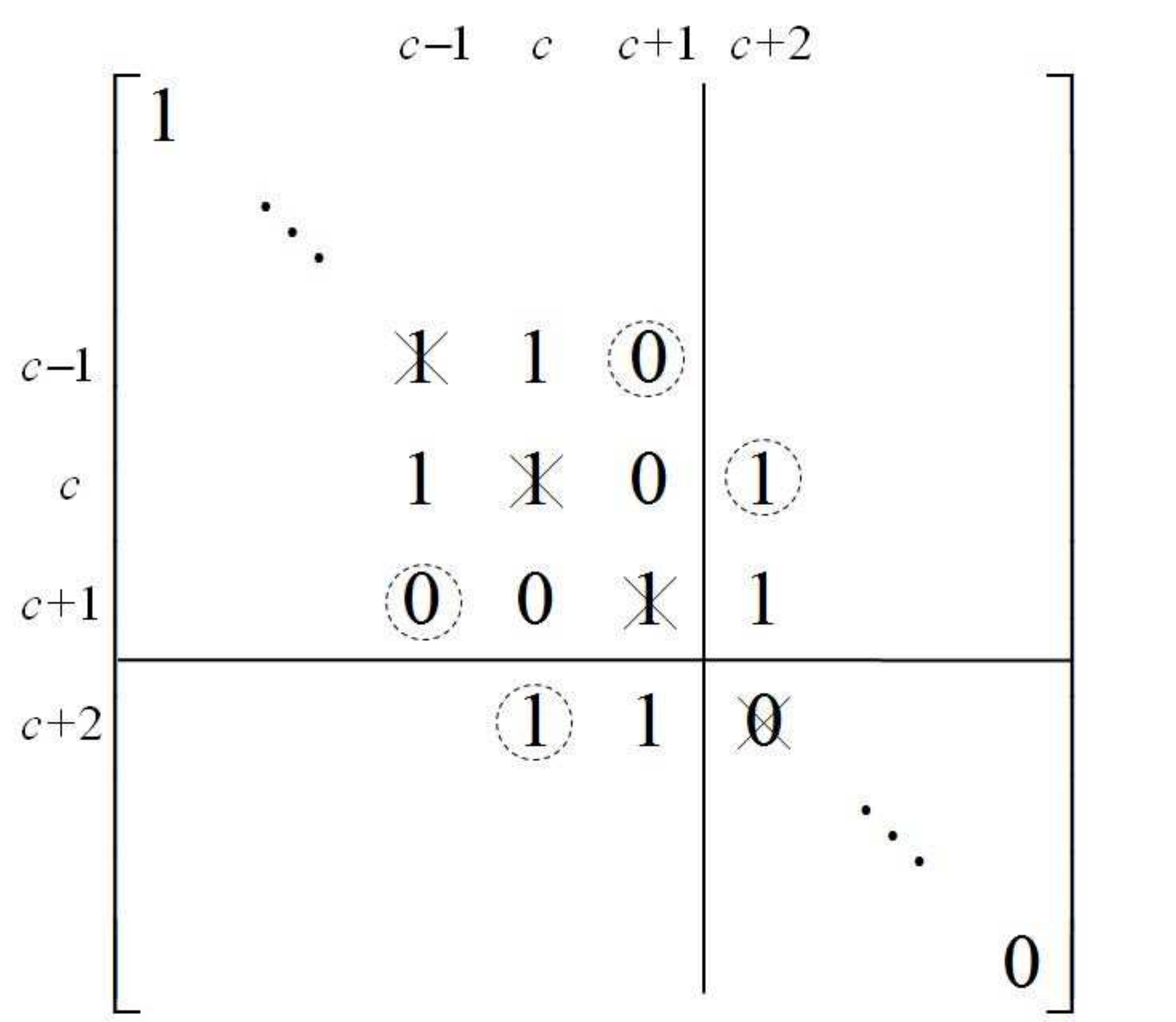}}
  \subfigure[]{\label{fig4b}\includegraphics[scale=0.27]{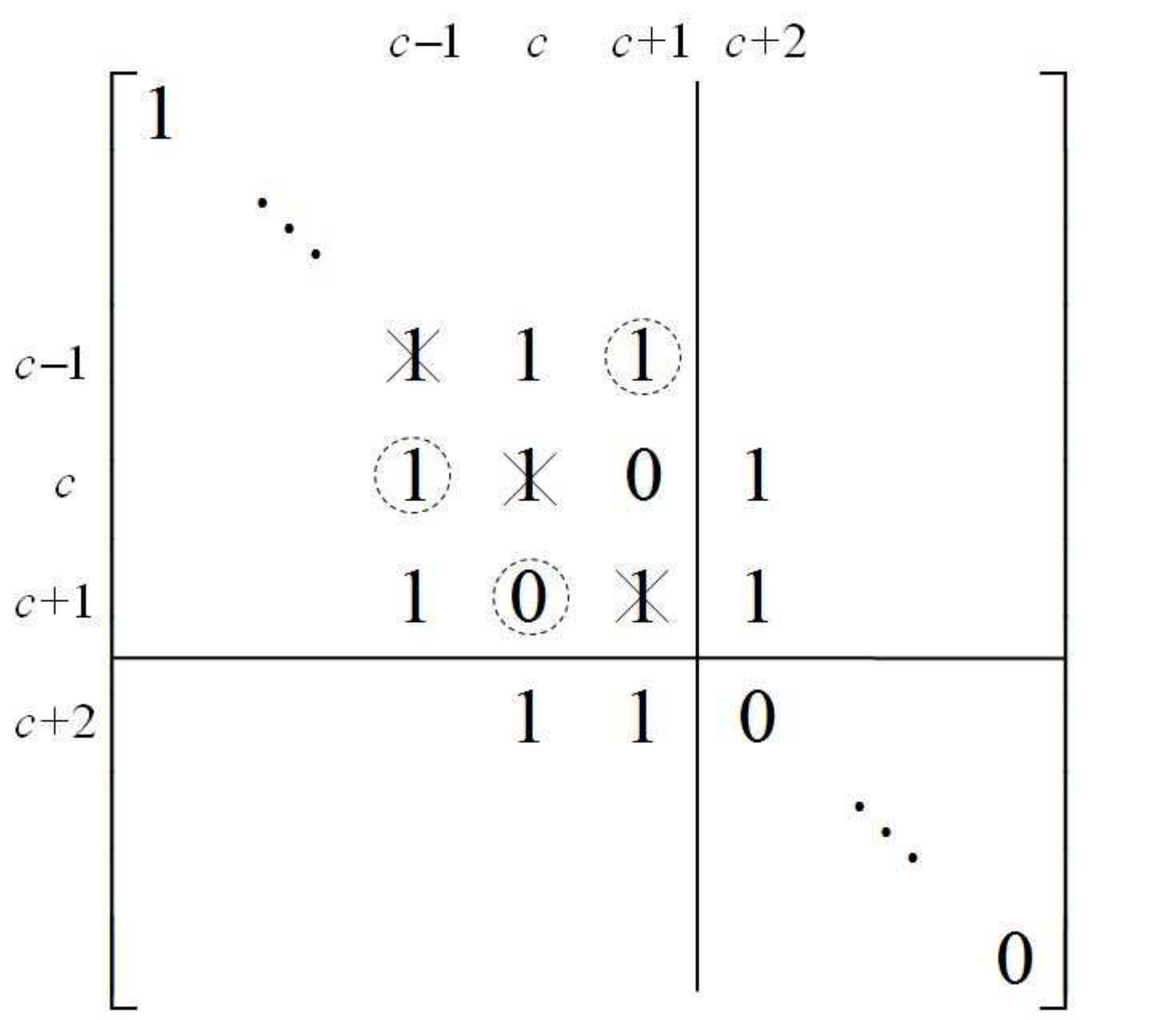}}
  \caption{}
  \label{fig4}
\end{figure}

\end{proof} 

For a matrix $K$ indexed by any set of indices $X$ and indices $i,j\in X$, denote by $K_{(i)}$ the row of $K$ indexed by $i$, and by $K^{(j)}$ the column of $K$ indexed by $j$.

\begin{claim}\label{claim:3}
For any  $j \in J$, the submatrix $A$ is the addition table
modulo 2 of the row $C_{(j)}$ and the column $B^{(j)}$
(See illustration in Figure~\ref{fig5}).
\end{claim}

\begin{figure}[h!]
\begin{center}
\includegraphics[scale=0.25]{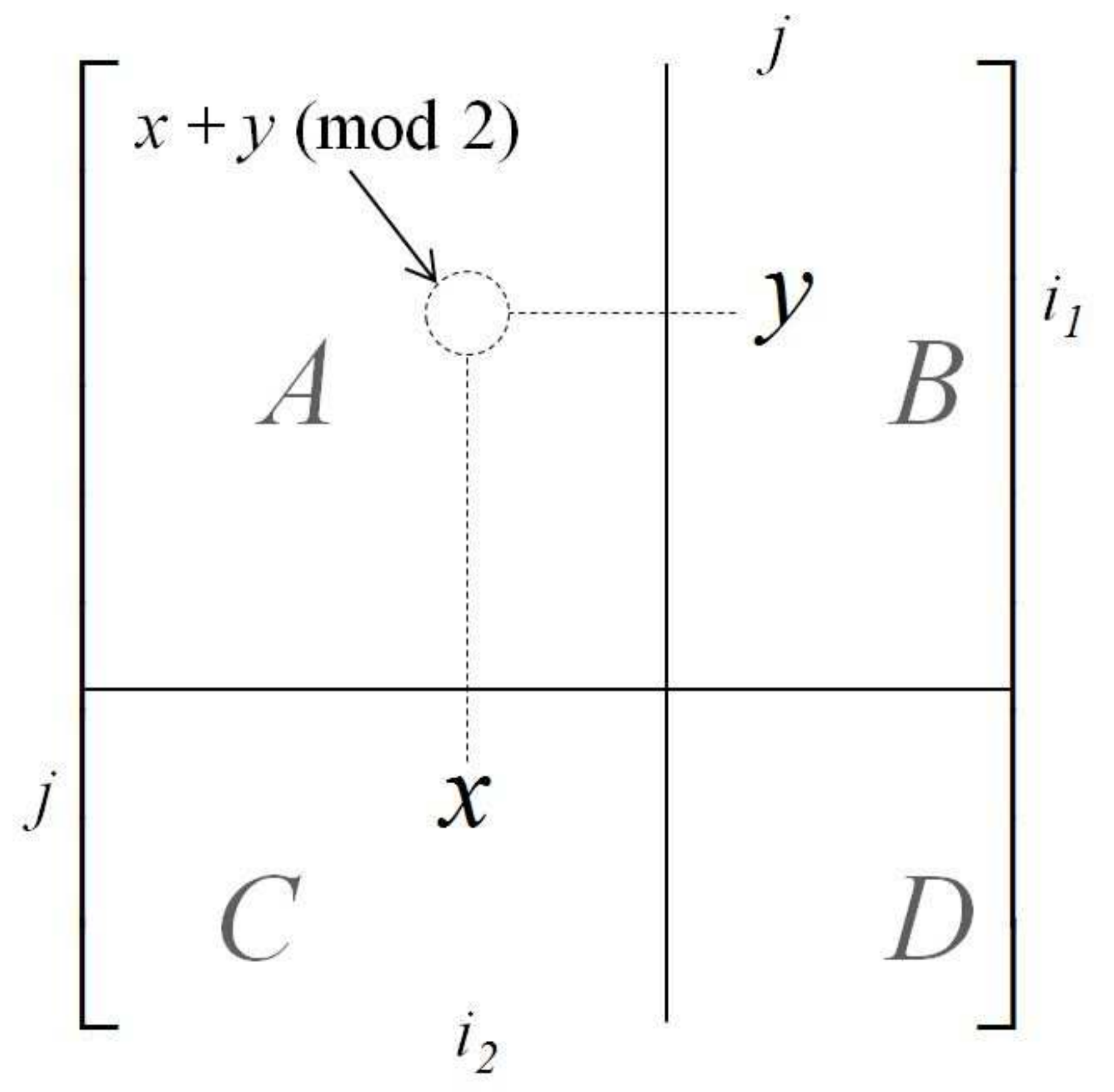}
\end{center}
\caption{}
\label{fig5}
\end{figure}

\begin{proof}[Proof of Claim~\ref{claim:3}]
\renewcommand{\qedsymbol}{}
We need to show that for any $i_1,i_2 \in I$ and $j \in J$ we have $m_{i_1,i_2}=m_{j, i_2}+m_{i_1,j} \pmod 2$. We may assume that $i_1\ne i_2$ since the case $i_1= i_2$ follows from Claim~\ref{claim:2} and the fact that $A$ has 1's in the main diagonal. Let $x=m_{j,i_2}\in C_{(j)}$ and $y=m_{i_1,j}\in B^{(j)}$. We consider three cases: (i) $x\ne y$, (ii) $x=y=0$, and (iii) $x=y=1$.

(i) Assume, for contradiction, that $m_{i_1,i_2}=0$. Then, by Claim 1, $m_{i_2,i_1}=0$ and we can replace $(i_1,i_1),(i_2,i_2)$ and $(j,j)$ in $T$ by $(i_2,i_1),(i_1,j)$ and $(j,i_2)$ and obtain a $c$-GD (Figure~\ref{fig6a}).
(ii) Assume, for contradiction, that $m_{i_1,i_2}=1$. We perform the same exchange as in Case (i) and, again, obtain a $c$-GD (Figure~\ref{fig6b}). (iii) By Claim 2, we have $m_{i_2,j}=m_{j,i_1}=0$. Assume, for contradiction, that $m_{i_1,i_2}=1$. We replace $(i_1,i_1),(i_2,i_2)$ and $(j,j)$ in $T$ by $(i_1,i_2),(i_2,j)$ and $(j,i_1)$ and obtain a $c$-GD (Figure~\ref{fig6c}). This proves Claim~\ref{claim:3}.

\begin{figure}[h!]
  \centering
  \subfigure[]{\label{fig6a}\includegraphics[scale=0.22]{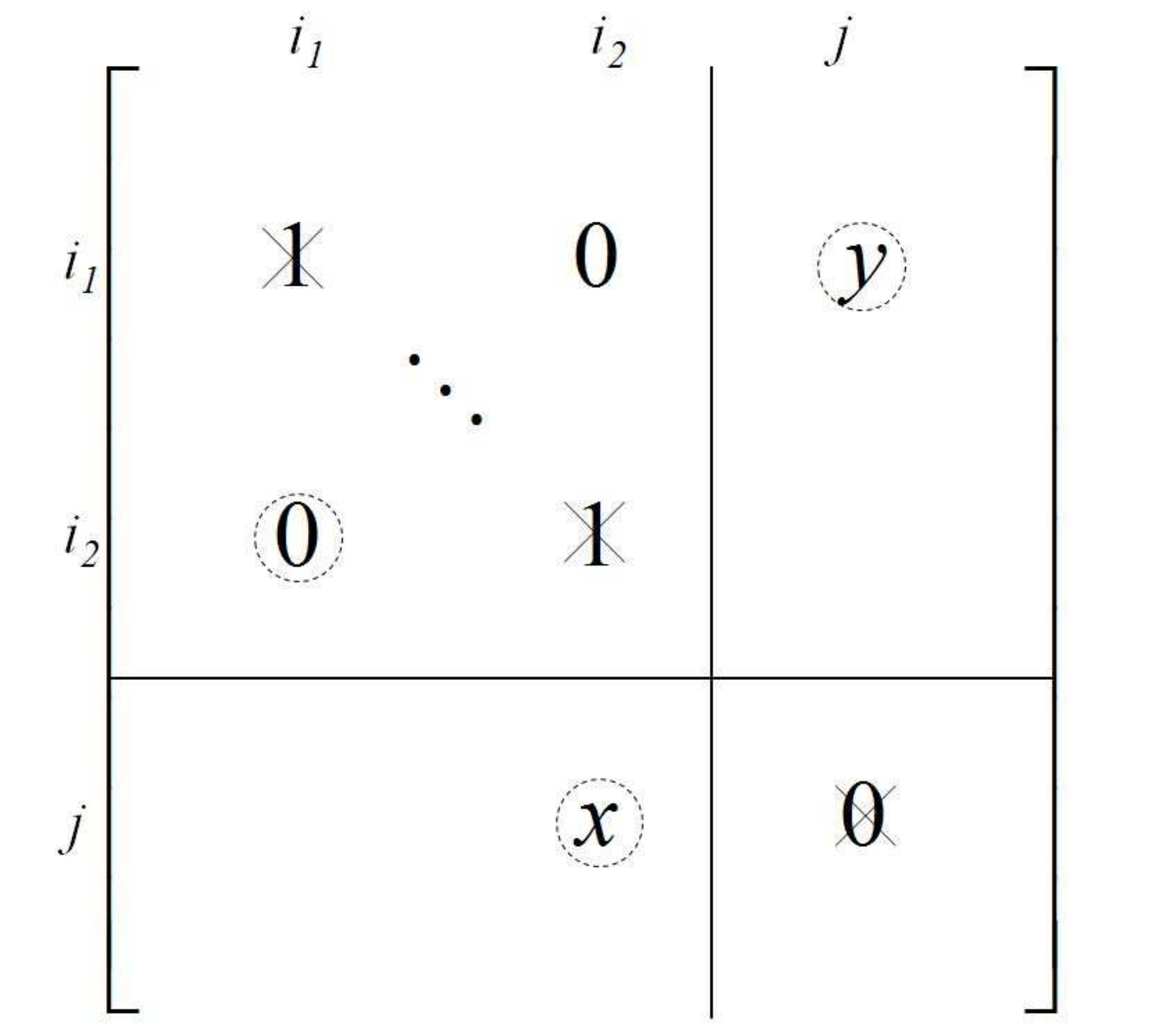}}
  \subfigure[]{\label{fig6b}\includegraphics[scale=0.22]{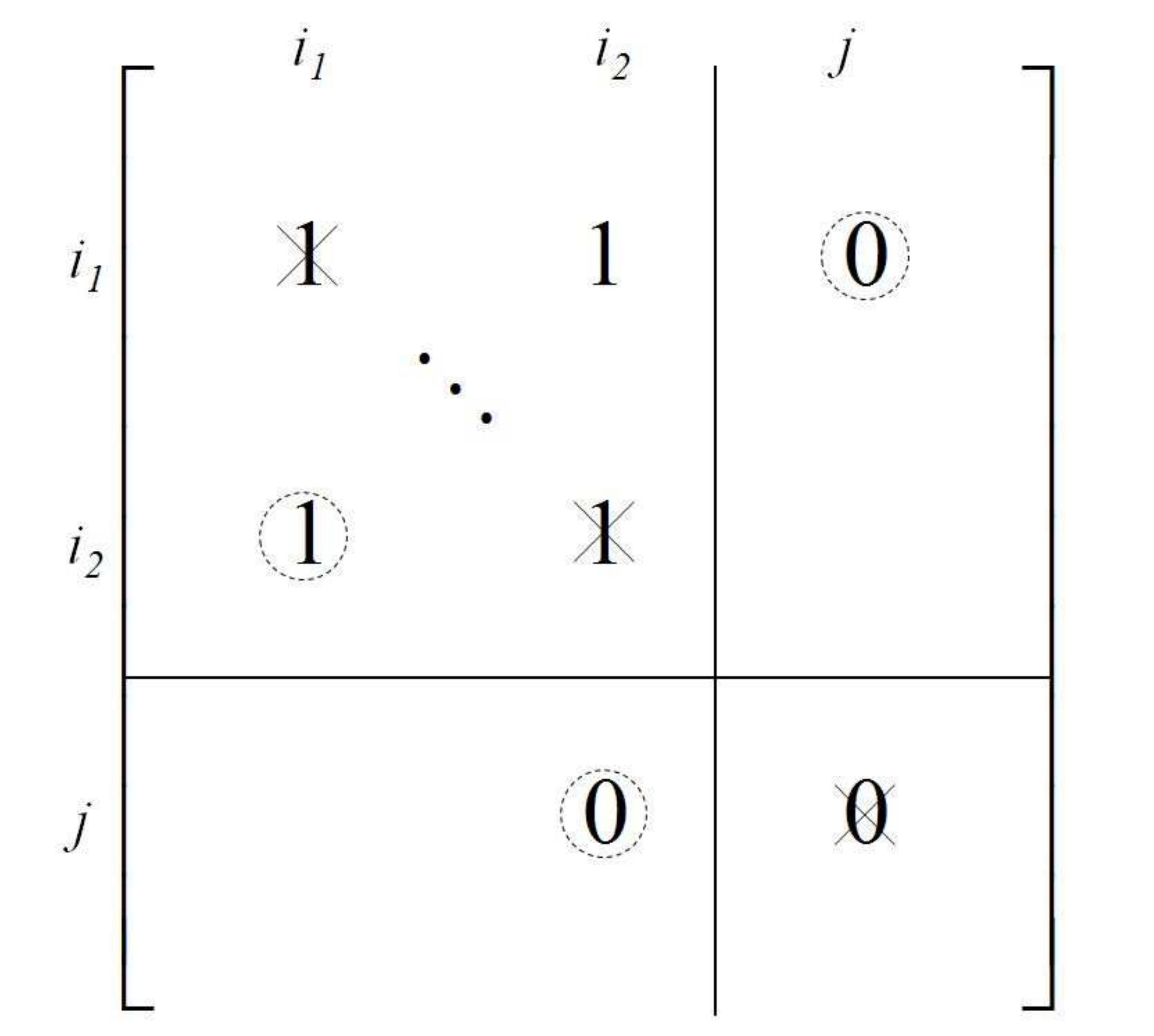}}
  \subfigure[]{\label{fig6c}\includegraphics[scale=0.22]{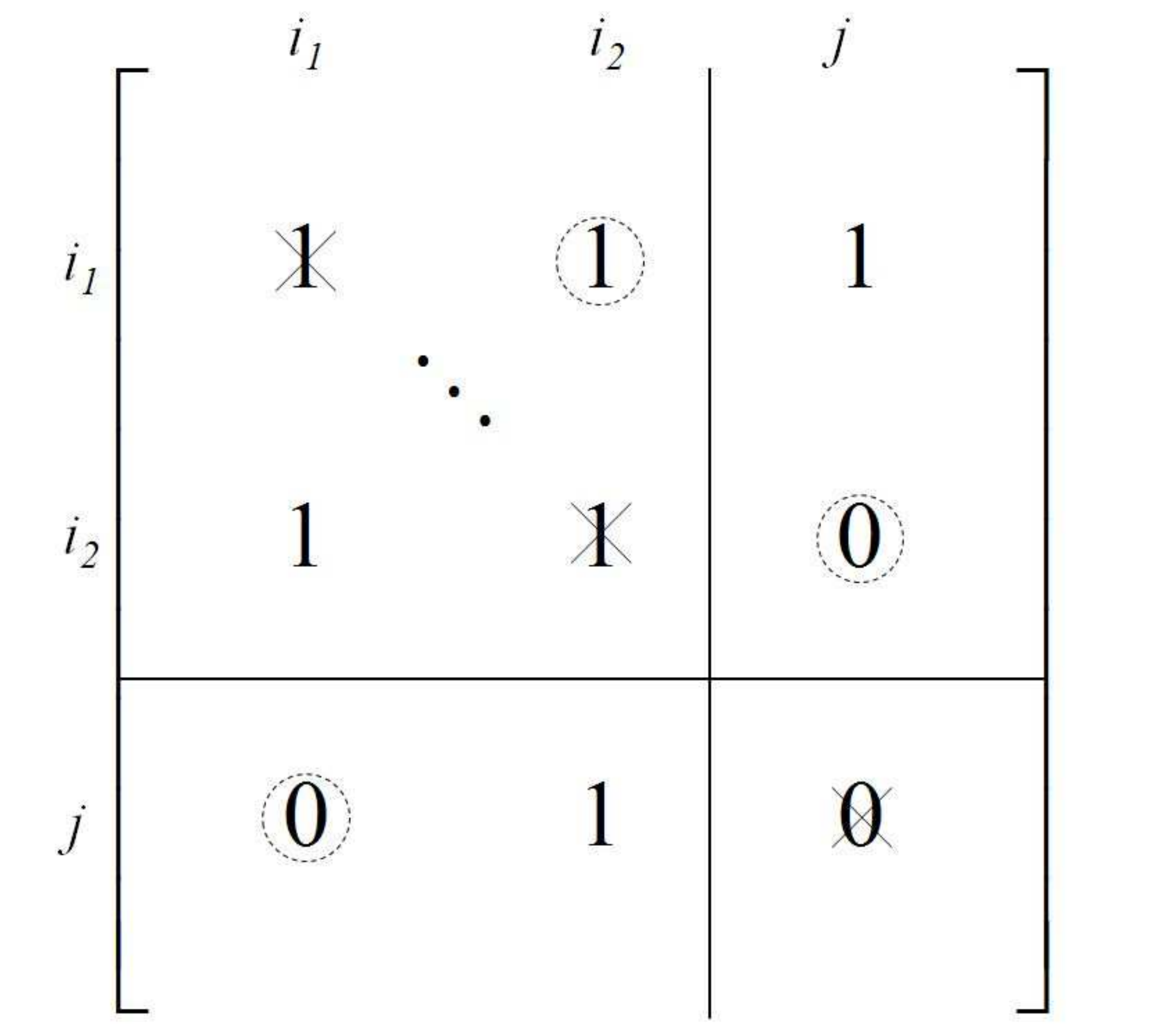}}
  \caption{}
  \label{fig6}
\end{figure}

\end{proof} 

We say that two (0,1)-vectors $u$ and $v$ of the same length are \emph{complementary} (denoted $u \bowtie v$) if their sum is the vector $(1,1,\ldots,1)$.  By Claim \ref{claim:3}, for every $i_1,i_2 \in I$, if  for some $j\in J$, it is true that $m_{i_1,j}=m_{i_2,j}$ then the two rows $A_{(i_1)},A_{(i_2)}$ are  identical, and if $m_{i_1,j}\ne m_{i_2,j}$ then these two rows are complementary. Furthermore - the rows $M_{(i_1)},M_{(i_2)}$ are identical or complementary. We summarize this in:

\begin{claim}\label{claim:4}
Any two rows in $M[I \mid [n]]$ are either identical or complementary.
\end{claim}

Next we show that the property in Claim~\ref{claim:4} holds for any two rows in $M$.

For $x,y\in \{0,1\}$ we define the operation $x\circ y=x+y+1 \pmod 2$ (Figure~\ref{fig7}).

\begin{claim}\label{claim:5}
The submatrix $D$ is the $\circ$-table between the column $C^{(i)}$ and the row $B_{(i)}$, for any $i\in I$.
\end{claim}

\begin{figure}[h!]
\begin{center}
\includegraphics[scale=0.25]{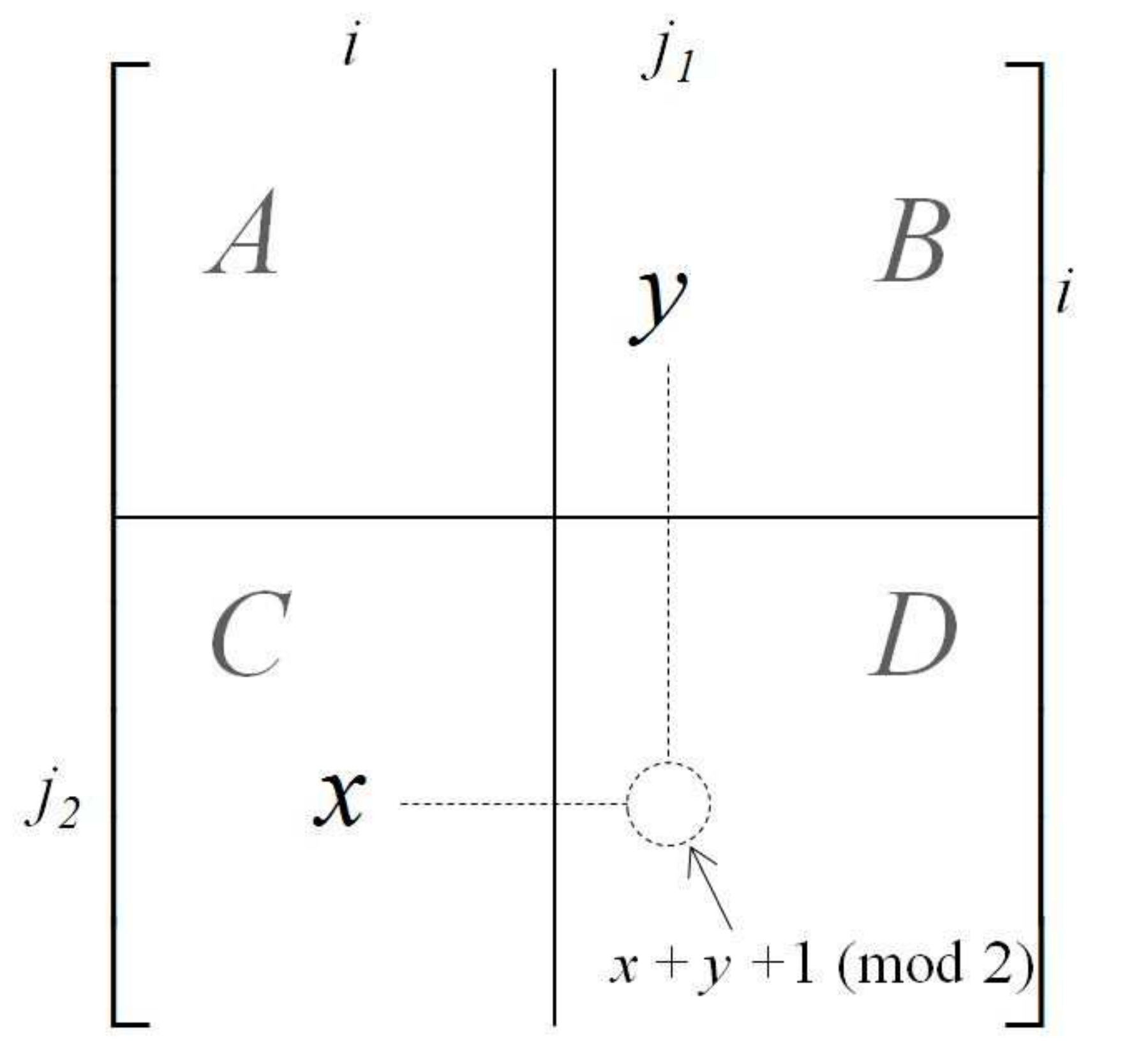}
\end{center}
\caption{}
\label{fig7}
\end{figure}

\begin{proof}[Proof of Claim~\ref{claim:5}]
\renewcommand{\qedsymbol}{}
We first consider $i$ such that $1\le i\le c-1$ (we assumed $c>1$). Let $j_1,j_2\in J$. We may assume that $j_1\ne j_2$ since the case $j_1=j_2$ follows from Claim~\ref{claim:2} and the fact that $D$ has 0's in the diagonal. Let $x=m_{j_2,i}$ and $y=m_{i,j_1}$. We consider three cases: (i) $x=y=0$, (ii) $x=y=1$, and (iii) $x\ne y$.

(i) Assume, for contradiction, that $m_{j_2,j_1}=0$. By Claim 1, $m_{j_1,j_2}=0$, and we can replace $(i,i),(j_1,j_1)$ and $(j_2,j_2)$ in $T$ by $(i,j_1),(j_1,j_2)$ and $(j_2,i)$ and obtain a $c$-GD (Figure~\ref{fig8a}). (ii) By Claim 2, $m_{j_1,i}=m_{i,j_2}=0$, and we can replace the same entries as in Case 1 by $(i,j_2)$, $(j_1,i)$ and $(j_2,j_1)$ and obtain a $c$-GD (Figure~\ref{fig8b}). (iii) Here is where we need the assumption $i\le c-1$. We perform the same replacement as in Case 1, but this time on the GD $T'$, and obtain a $c$-GD (Figure~\ref{fig8c}. Recall that $T'$ is a $(c-1)$-GD).

\begin{figure}[h!]
  \centering
  \subfigure[]{\label{fig8a}\includegraphics[scale=0.22]{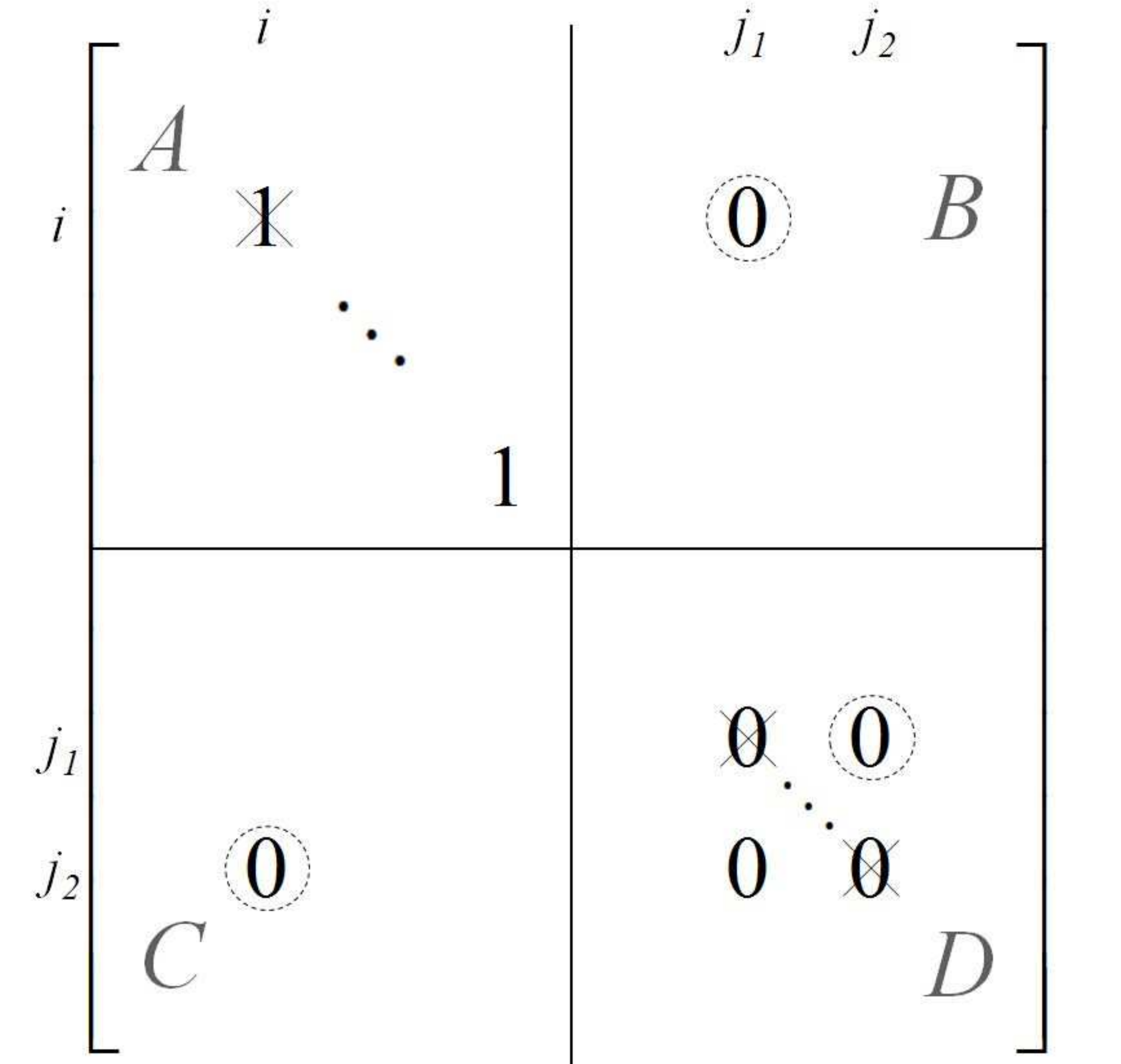}}
  \subfigure[]{\label{fig8b}\includegraphics[scale=0.22]{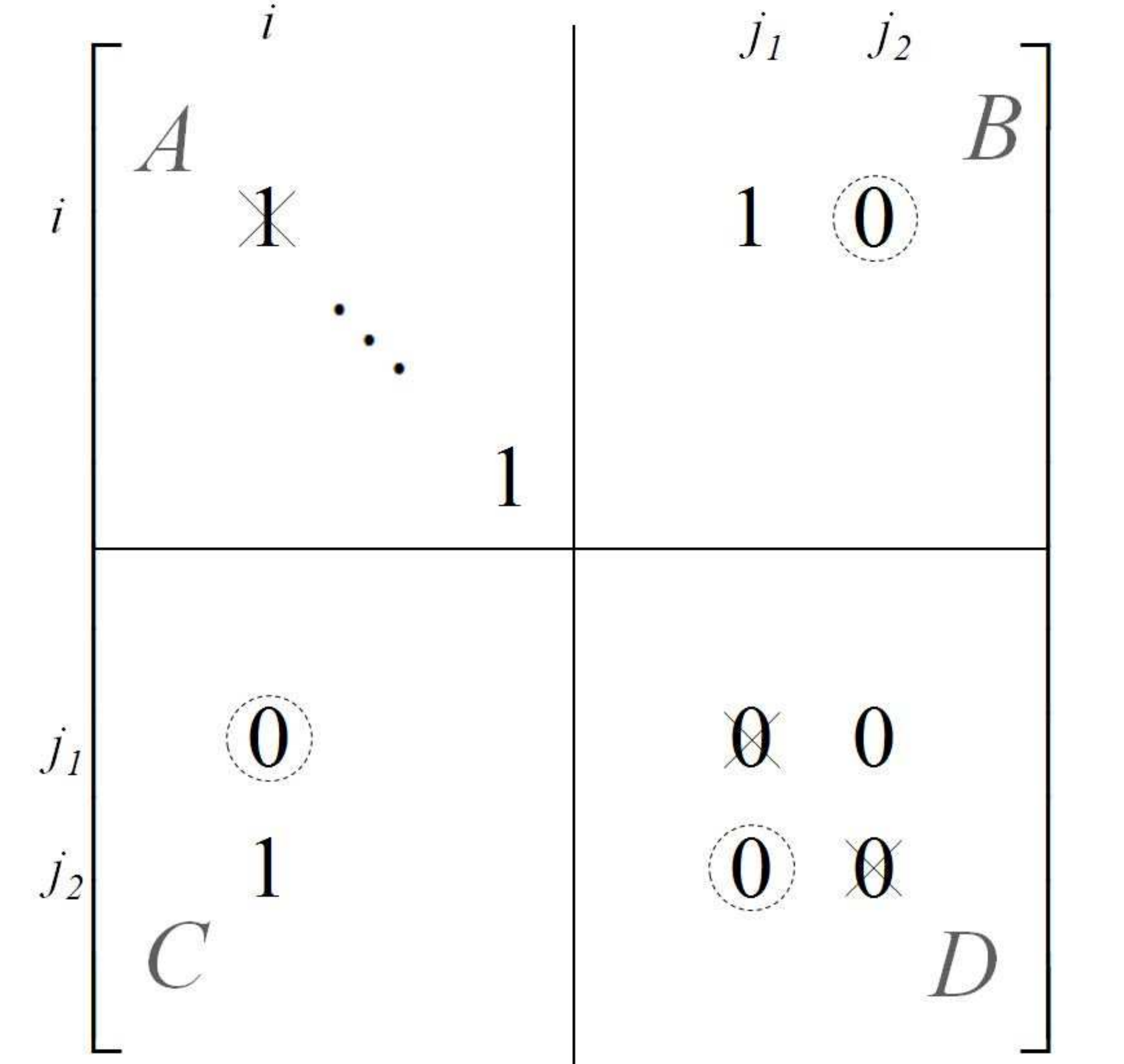}}
  \subfigure[]{\label{fig8c}\includegraphics[scale=0.22]{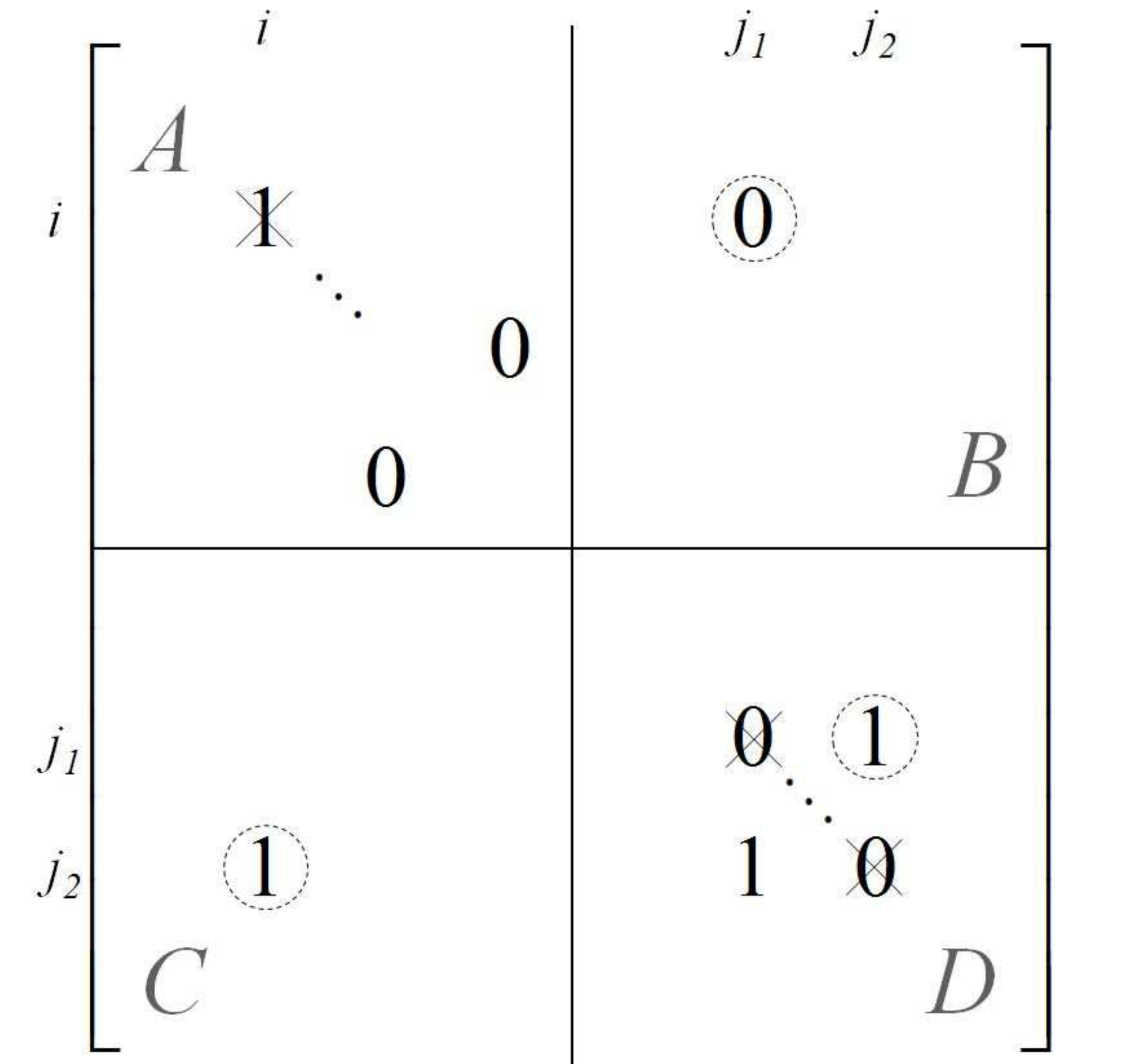}}
  \caption{}
  \label{fig8}
\end{figure}

It remains to prove the claim for $i=c,c+1$. It follows from Claim \ref{claim:4} that any two rows of $B$ are either identical or complementary. Thus, by Claim~\ref{claim:2}, any two columns of $C$ are either identical or complementary. If there exists $j<c$ such that $B_{(c)} = B_{(j)}$, then $C^{(c)} = C^{(j)}$. Since $D$ is the $\circ$-table between $C^{(j)}$ and $B_{(j)}$, it is also the $\circ$-table between $C^{(c)}$ and $B_{(c)}$. If all $j<c$ satisfy $B_{(c)} \bowtie B_{(j)}$, then for any such $j$, we have $C^{(c)} = B_{(j)}^T$ and $C^{(j)} =B_{(c)}^T$ by Claim~\ref{claim:2}. Since $\circ$ is commutative we again have that $D$ is the $\circ$-table between $C^{(c)}$ and $B_{(c)}$. A similar argument holds for $i=c+1$.
\end{proof} 

\begin{claim}\label{claim:6}
Any two rows of $M$ are either identical or complementary.
\end{claim}

\begin{proof}[Proof of Claim~\ref{claim:6}]
\renewcommand{\qedsymbol}{}
The fact that any two rows in $M[J|[n]]$ are either identical or complementary follows in the same manner as Claim~\ref{claim:4}.
Now, assume $i\in I, j\in J$. We want to show that $M_{(i)}$ is either identical or complementary to $M_{(j)}$.
From Claim~\ref{claim:3} we know that $A_{(i)}$ is either identical or complementary to $C_{(j)}$ and from Claim~\ref{claim:5} we have that $B_{(i)}$ is either identical or complementary to $D_{(j)}$.
We need to show that $A_{(i)}$ is identical to $C_{(j)}$ if and only if $B_{(i)}$ is identical to $D_{(j)}$. Note that $m_{ii}=1$, $m_{jj}=0$ and $m_{ij}\ne m_{ji}$. So, if $m_{ji}=1$ we have identity in both cases and if $m_{ji}=0$ we have complementarity in both cases.
\end{proof} 

Suppose all the rows of $M$ are identical. Then, the first $c+1$ columns are all-1 and the rest of the columns are all-0. So, any GD has exactly $c+1$ 1s. So, $a=b=c+1$, which is obviously not the case. Thus, by Claim~\ref{claim:6}, we can permute the rows and columns to obtain a matrix $M'$ consisting of four submatrices $M_1,M_2,M_3$ and $M_4$ of positive dimensions, where $M_1$ and $M_4$ are all-1, and $M_2$ and $M_3$ are all-0 (Figure~\ref{fig9}).

\begin{figure}[h!]
\begin{center}
\includegraphics[scale=0.25]{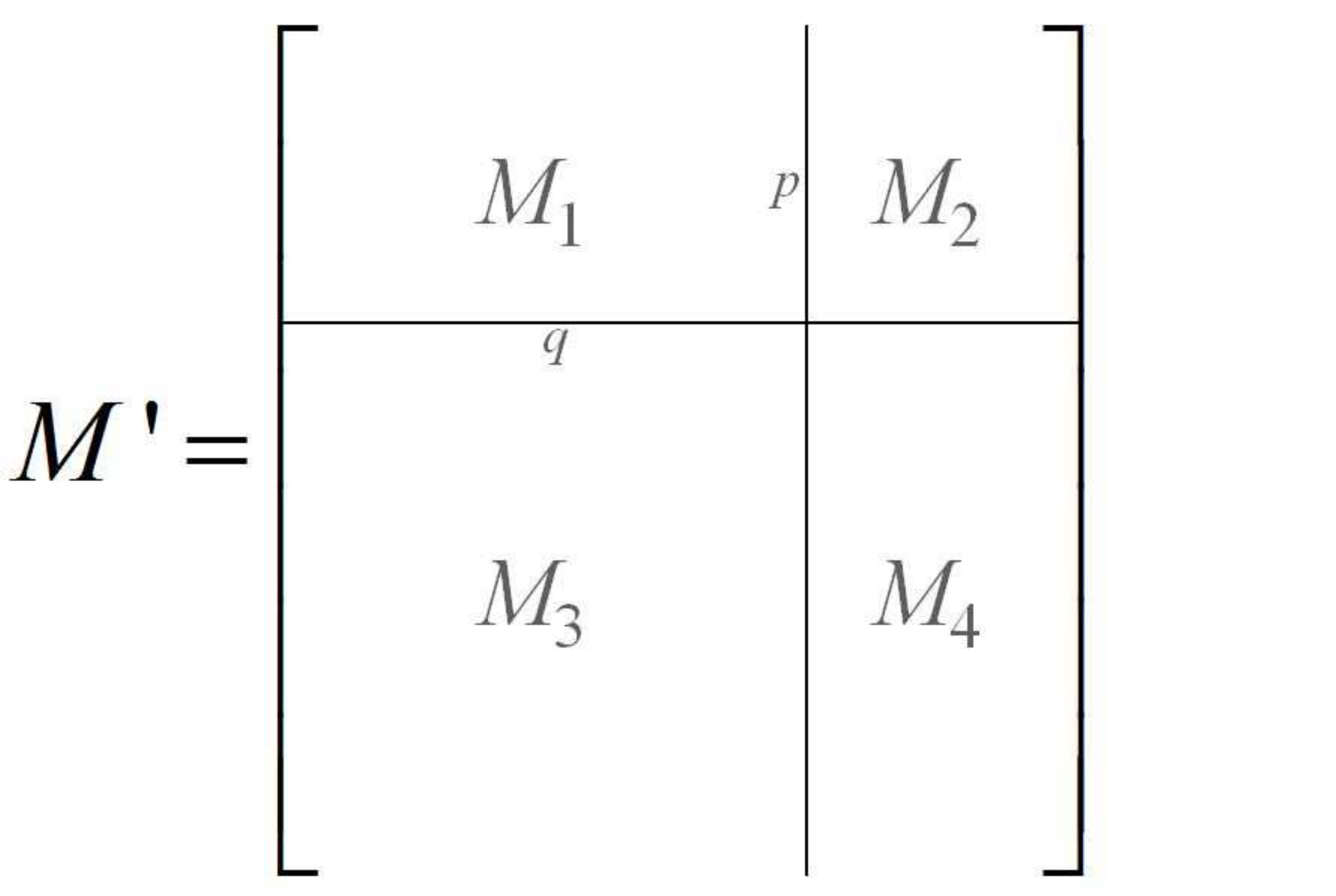}
\end{center}
\caption{}
\label{fig9}
\end{figure}

Thus, $F$ is rigid (Definition~\ref{def:rigidity}), contrary to the hypothesis. We conclude that there must be a $c$-GD in $M$.
\end{proof} 

In the case that the partition $E(G)=F\cup (E(G)\setminus F)$ is rigid, if there exists a partition $P_{c+1}$ such that $|P_{c+1}\cap F|=c+1$, then clearly there is no partition $P_c$ such that $|P_{c}\cap F|=c$. The proof of Theorem~\ref{theorem:d} shows that in this case, for any $c$ between $a$ and $b$ there is a partition $P_{c'}$ such that $0\le|P_{c'}\cap F|-c\le 1$.

\begin{corollary}\label{corollary:d}
Let $G=K_{n,n}$ and assume the partition $E(G)=F\cup (E(G)\setminus F)$ is not rigid. Then, there exist perfect matchings $P_1$ and $P_2$ such that $|P_1\cap F|= \floor*{\frac{|F|}{n}}$ and $|P_2\cap F|= \ceil* {\frac{|F|}{n}}$.
\end{corollary}

\section{Fair representation by perfect matchings in $K_{n,n}$, the case of three parts}
In this section we prove Conjecture \ref{equirep00} for $m=3$, namely:

\begin{theorem}\label{equirep=3}
Suppose that  the edges of $K_{n,n}$ are partitioned into
 sets $E_1, E_2, E_3$. Then,
there exists a perfect matching
$F$ in $K_{n,n}$ satisfying $\ceil*{ \frac{|E_i|}{n}}+1 \ge |F \cap E_i| \ge \floor*{ \frac{|E_i|}{n}
 } -1$ for every $i= 1, 2, 3$.
\end{theorem}

It clearly suffices to prove the theorem for  partitions of $E(K_{n, n})$ into sets $E_1, E_2, E_3$ such that
$|E_i| = k_i n$,  for $k_i$  integers $(i=1,2,3)$.
Assuming negation of Theorem \ref{equirep=3}
 there is no perfect matching
with exactly $k_i$ edges from each $E_i$.
As already mentioned, the theorem is patently true if one of the sets $E_i$ is empty,
so we may assume $k_1, k_2, k_3 \in \{1, \ldots, n-2\}$.

We  identify perfect matchings in $K_{n,n}$ with permutations in $S_n$.
For  $\sigma, \tau \in S_n$, the {\em Hamming distance} (or plainly {\em distance}) $d(\sigma, \tau)$ between $\sigma$ and $\tau$ is
 $|\{i \mid \sigma(i) \neq \tau(i)\}|$.
We write $\sigma \sim \tau$ if $d(\sigma, \tau) \le 3$. Let $\cc$ be the simplicial complex of the cliques of this relation.
So, the vertices of $\cc$ are the permutations in $S_n$ and the simplexes are the sets of permutations each two of which have distance at most 3 between them.
The core of the proof of the theorem will be in showing that $\cc$ is simply connected, which will enable us to use Sperner's lemma.

Here is a short outline of the proof of the theorem.
Clearly, for each $i \le 3$ there exits a matching $F_i$  representing $E_i$ fairly, namely $|F_i \cap E_i| \ge \floor*{ \frac{|E_i|}{n} }$.
We shall  connect every pair $F_i,F_j~~(1\le i <j\le 3)$ by a path consisting of perfect matchings representing fairly $E_i \cup E_j$, in such a way that every two adjacent matchings are $\sim$-related. This  generates a triangle $D$ that is not necessarily simple (namely it may have repeating vertices), together with a triangulation $T$ of its circumference, and an assignment $A$ of matchings to its vertices. We shall then show that there exists a triangulation $T'$ extending $T$ and contained in $\cc$ (meaning that there is an assignment $A'$ extending $A$ of perfect matchings to the vertices of $T'$), such that the perfect matchings assigned to adjacent vertices are $\sim$-related. We color a vertex $v$ of $T'$ by color  $i$ if $A'(v)$ represents fairly the set $E_i$. By our construction, this coloring satisfies the conditions of the $2$-dimensional version of Sperner's lemma, and applying
 the lemma we obtain a multicolored triangle. We shall then show that at least one of the matchings assigned to the vertices of this triangle satisfies the condition required in the theorem. \\

\subsection{Topological considerations}

 Let us recall the $2$-dimensional version of Sperner's lemma:

\begin{lemma}
\label{originalsperner}
Let $T$ be a triangulation of a triangle $ABC$ and suppose that the vertices of $T$ are colored $ 1,2,3$.
Assume that
\begin{itemize}
\item The vertex A has color 1.
\item The vertex B has color 2.
\item The vertex C has color 3.
\item Every vertex in the subdivision of the edge AB has either color 1 or color 2.
\item Every vertex in the subdivision of the edge BC has either color 2 or color 3.
\item Every vertex in the subdivision of the edge CA has either color 3 or color 1.
\end{itemize}
Then $T$ contains a region triangle with three vertices colored 1, 2 and 3.
\end{lemma}

We shall need a ``hexagonal'' version of the lemma:

\begin{lemma}
\label{hexagonsperner}
Let  $T$ be a triangulation of a hexagon, whose outer cycle is the union of six paths $p_1, \ldots, p_6$
(which are, in a cyclic order, subdivisions of the six edges of the hexagon). Suppose that the vertices of $T$ are colored $ 1,2,3$, in such a way that
\begin{itemize}
\item No vertex in $p_1$ has color 1.
\item No edge in $p_2$ is between two vertices of colors 1 and 2.
\item No vertex in $p_3$ has color 2.
\item No edge in $p_4$ is between two vertices of colors 2 and 3.
\item No vertex in $p_5$ has color 3.
\item No edge in $p_6$ is between two vertices of colors 3 and 1.
\end{itemize}
Then $T$ contains a region triangle with three vertices colored 1, 2 and 3.
\end{lemma}

\begin{proof}
 Add three vertices to $T$ outside the circumference of the hexagon in the following way.
Add a vertex $A$ of color 1 adjacent to all vertices in $p_4$,
a vertex $B$ of color 2 adjacent to all vertices in $p_6$
and a vertex of color 3 adjacent to all vertices in $p_2$.
Using Sperner's Lemma on this augmented triangulation yields the lemma.
\end{proof}

Our strategy for the proof of Theorem \ref{equirep=3} is the following. First we form a triangulation of a hexagon and
assign a permutation in $S_n$ to each vertex of the triangulatin, where adjacent permutations are $\sim$ related. Afterwards we
color each permutation $\sigma$ with some color $i$, where $E_i$ is fairly represented in $\sigma$.
We then apply Lemma \ref{hexagonsperner} to get three permutations $\sigma_1, \sigma_2, \sigma_3$ which are pairwise $\sim$ related,
and fairly represent $E_1, E_2, E_3$ respectively. We then show that how to use this to construct a permutation
  almost fairly representing all three sets $E_1, E_2, E_3$, simultaneously.

For $i \in [n]$ let
$shift_{i} : S_n \to S_n$ be a function defined as follows.
For every $\sigma \in S_n$, if $\sigma(i) = j$ then

$$shift_i(\sigma)(k) = \left\{
\begin{array} {rcl}
i & \mbox{if }  k=i \\
j & \mbox{if } \sigma(k) = i \\
\sigma(k) & \mbox{otherwise}
\end{array}
\right.
$$


\begin{remark}
Note that if $\sigma(i)=i$ then $shift_i(\sigma)=\sigma$.
\end{remark}


\begin{lemma}
If  $\sigma \sim \tau$ then $shift_i(\sigma) \sim shift_i(\tau)$.
\end{lemma}

\begin{proof}
 Without loss of generality let $i= 1$. If $shift_1(\sigma) = \sigma$ and $shift_1(\tau) = \tau$ then we are done.

{\bf Case I:} $shift_1(\tau) = \tau$ and
$shift_1(\sigma) \neq \sigma$. Without loss of generality $\tau = I$, the identity permutation.
For every $k \in [n]$, if $\sigma(k)=k$ then also $shift_1(\sigma)(k)=k$ and thus
the distance between $shift_1(\sigma)$ and $I$ is at most the distance between $\sigma$ and $I$,
yielding $shift_1(\sigma) \sim I = shift_1(\tau)$.

{\bf Case II:} $shift_1(\sigma) \neq \sigma$ and $shift_1(\tau) \neq \tau$.
Without loss of generality
$\tau = (1 2)$ and hence $shift_1(\tau) = I$.
As in the previous case, for every $k \in [n]$ if $\sigma(k)=k$ then also $shift_1(\sigma)(k)=k$.
We also note that $shift_1(\sigma)(1) = 1$ but $\sigma(1) \ne 1$ (since $shift_1(\sigma) \neq \sigma$).
Therefore $d(shift_1(\sigma),I)< d(\sigma,I)$.
If $d(\sigma,I)\le 4$
 then $shift_1(\sigma) \sim I = shift_1(\tau)$ and we are done.
Since $\sigma \sim \tau$, we have $d(\sigma,I)\le 5$  so we may assume that $d(\sigma,I)=5$. Note that if $\sigma(1)=j\ne 2$, then $\sigma$ and $\tau$ differ on 1,2 and $j$, and thus $\sigma(k)=k$ for all $k\not\in\{1,2,j\}$, so $d(\sigma,I) \le 3$, contrary to the assumption that this distance is $5$. Thus, we must have that $\sigma(1)=2$.
It follows that   $A:= \{i \in [n] : \sigma(i) \neq \tau(i)\}$
is a set of size $3$ disjoint from $\{1,2\}$. But then also  $\{i \in [n] : shift_1(\sigma(i) )\neq shift_1(\tau(i))\} = A$,
yielding $shift_1(\sigma) \sim shift_1(\tau)$.
\end{proof}

At this point we need a connectivity result. This is best formulated in  matrix language.

\begin{lemma}\label{largepermutationmatrices}
Let $A = (a_{ij})$ be an $n \times n$ 0-1 matrix and let $k \in [n-1]$. Let $G$ be the graph whose
vertices are the permutations $\sigma \in S_n$ satisfying $\sum_{i=1}^n a_{i \sigma(i)} \geq k$ and whose edges correspond to the $\sim$ relation. If there exists $\rho \in S_n$ with $\sum_{i=1}^n a_{i \rho(i)} > k$, then $G$ is connected.
\end{lemma}

\begin{proof}
Without loss of generality $\rho=I$, meaning that $\sum_{i=1}^n a_{ii} > k$.
 We shall show that there is a path in $G$ from $\rho$ to  $\sigma$ for any $\sigma \in V(G) \setminus \{\rho\}$.
We prove this claim by induction on  $d(\sigma,\rho)$.
Write $\ell = \sum_{i=1}^n a_{i \sigma(i)} $.
Our aim  is to find distinct $j \in [n]$ for which $\sigma(j) \neq j$ and
$\sigma' = shift_{j}(\sigma) \in V(G)$. Then the induction hypothesis can be applied since $\sigma \sim \sigma'$ and $\sigma'$ is closer to $\rho$ than $\sigma$.\\

If $\ell \geq k+2$  choose any $j \in [n]$ with $\sigma(j) \neq j$. Then we have
$\sum_{i=1}^n a_{i \sigma'(i)} \geq \sum_{i=1}^n a_{i \sigma(i)} - 2 \geq k$, so $\sigma' \in V(G)$.

Suppose next that $\ell = k+1$. By the assumption that $\sum_{i=1}^n a_{ii} > k$ we have $\sum_{i=1}^n a_{i \sigma(i)} \leq \sum_{i=1}^n a_{i i}$ and since $\sigma \neq \rho$
there must be some $j \in [n]$ for which $\sigma(j) \neq j$ and $a_{j j} \geq a_{j \sigma(j)}$.
Taking $\sigma' = shift_{j}(\sigma) \in V(G)$ yields
$\sum_{i=1}^n a_{i \sigma'(i)} \geq \sum_{i=1}^n a_{i \sigma(i)} - 1 = k$, so $\sigma' \in V(G)$.

Finally, if $\ell = k$ then $\sum_{i=1}^n a_{i \sigma(i)} < \sum_{i=1}^n a_{i i}$ and hence
there must be some $j \in [n]$ for which $a_{j j} > a_{j \sigma(j)}$.
Taking $\sigma' = shift_{j}(\sigma) \in V(G)$ we get
$\sum_{i=1}^n a_{i \sigma'(i)} \geq \sum_{i=1}^n a_{i \sigma(i)} + 1 - 1= k $, so $\sigma' \in V(G)$.
\end{proof}

\begin{corollary}
\label{path}
Let $A = (a_{ij})$ be an $n \times n$ 0-1 matrix and let $k \in [n]$. Let $G$ be the graph whose
vertices are the permutations $\sigma \in S_n$ with $\sum_{i=1}^n a_{i \sigma(i)} \geq k$ and whose edges correspond to the $\sim$ relation. If $\sum_{i, j \le n}a_{ij} \geq kn$ then $G$ is connected.
\end{corollary}

\begin{proof}
If there exists a permutation $\rho$ with $\sum_{i=1}^n a_{i \rho(i)} > k$ then we are done by Lemma \ref{largepermutationmatrices}. If not, by K\"onig's theorem there exist sets
$A, B \subseteq [n]$ with $|A|+|B|\le k$ such that $a_{ij}=0$ for  $i \not \in A$ and $j \not \in B$. This is compatible with the condition $\sum_{i, j \le n}a_{ij} \geq kn$  only if $|A|=0$ and $|B|=k$ or $|B|=0$ and $|A|=k$, and
$a_{ij}=1$ for all $(i,j) \in A\times [n] \cup [n]\times B$. In both cases  $V(G)=S_n$, implying that the relation $\sim$ is path connected since every permutation is reachable from every other permutation by a sequence of transpositions.
\end{proof}

 In the next two lemmas let $i\in[n]$ and $\sigma,~\tau\in S_n$. We write $shift$ for $shift_i$.

\begin{lemma}
If $d(\sigma, \tau)=2$, then the 4-cycle $\sigma-\tau-shift(\tau)-shift(\sigma)-\sigma$ is null-homotopic in ${\mathcal C}$
(i.e., it can be triangulated.)
\end{lemma}

\begin{proof}
If either $\sigma \sim shift(\tau)$ or $\tau \sim shift(\sigma)$ then we are done.
So, we may assume this does not happen and in particular $\sigma \neq shift(\sigma)$
and $\tau \neq shift(\tau)$. We may assume, without loss of generality, that $i=1$,
$\sigma = (1 2)$, $\tau = (1 2)(3 4)$, $shift(\sigma) = I$ and $shift(\tau) = (3 4)$.
We can now fill the cycle as in Figure~\ref{fig10}.
\end{proof}

\begin{figure}[h!]
\begin{center}
\label{simpleC4}
\includegraphics[scale=0.3]{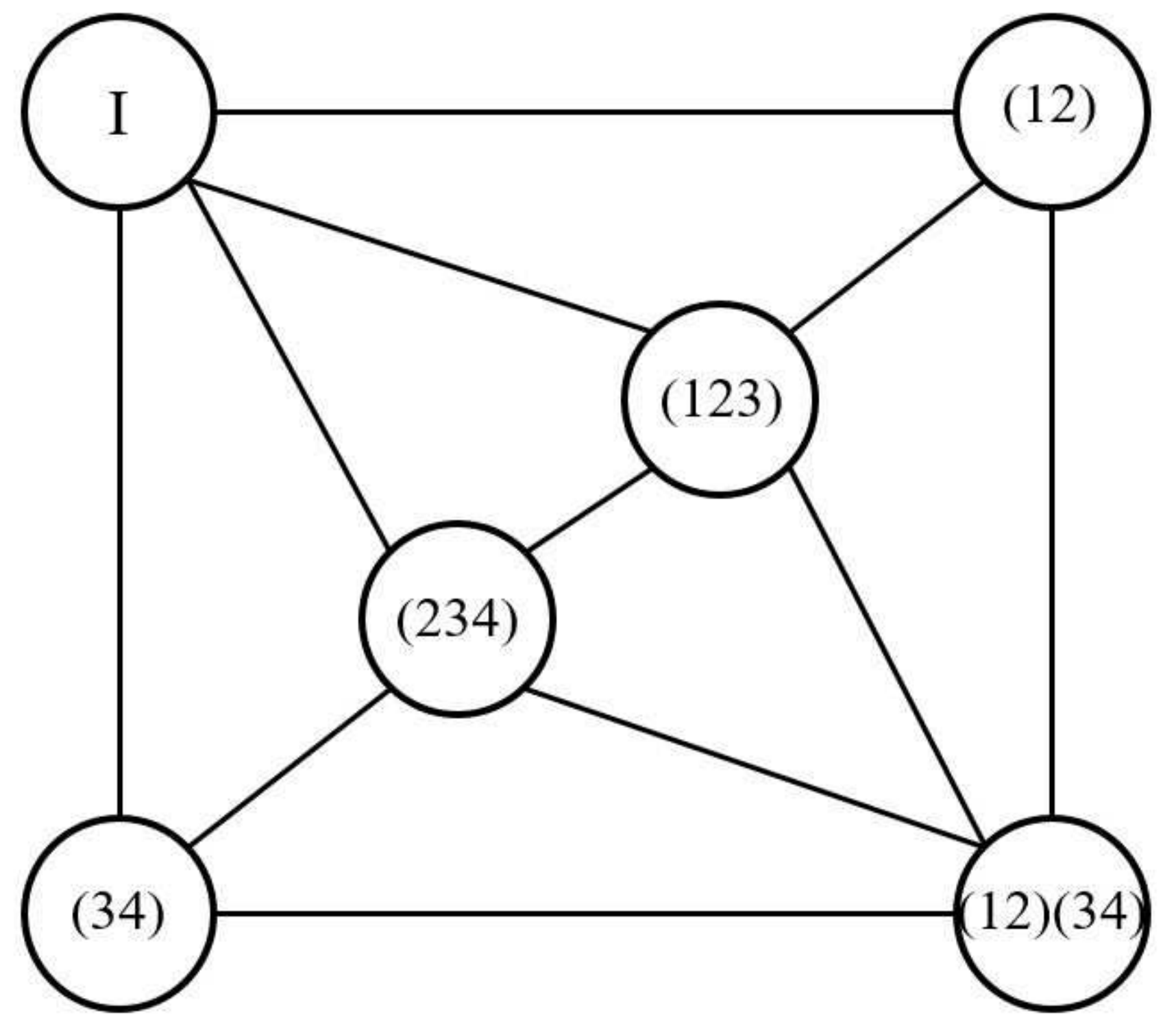}
\end{center}
\caption{}
\label{fig10}
\end{figure}

\begin{lemma}
If  $d(\sigma, \tau) = 3$
 then the 4-cycle $\sigma-\tau-shift(\tau)-shift(\sigma)-\sigma$ is a null cycle in ${\mathcal C}$.
\end{lemma}

\begin{proof}
Let $\rho \in S_n$ have distance 2 from both $\sigma$ and $\tau$. Denote $\sigma'=shift(\sigma)$, $\tau'=shift(\tau)$ and $\rho'=shift(\rho)$. We use the previous lemma to fill the cycle as in Figure~\ref{fig11}.
\end{proof}

\begin{figure}[h!]
\begin{center}
\label{dist3C4}
\includegraphics[scale=0.3]{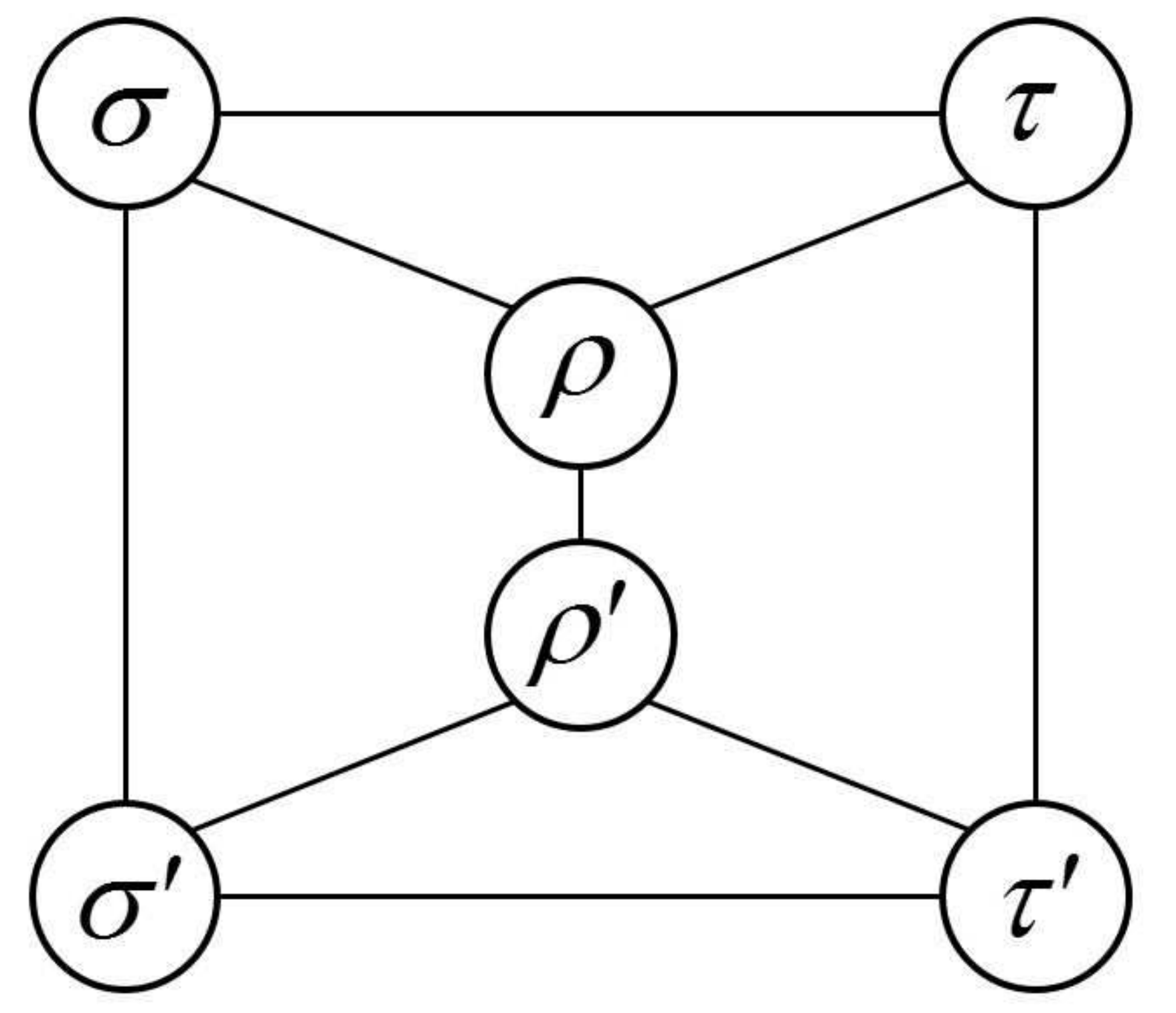}
\end{center}
\caption{}
\label{fig11}
\end{figure}

As a corollary from the above two lemmas we get

\begin{corollary}\label{cor:simplicial:map}
Let $C$ be a cycle  and let $f : C \to \cc$ be a simplicial map, i.e., mapping each edge to an edge or a vertex.
Let $\bar{f}: C \to \cc$ be defined by $\bar{f}(v) = shift_i(f(v))$ for every $v \in V(C)$.
Then $\bar{f}$ is also simplicial and is homotopic to $f$.
\end{corollary}

See Figure~\ref{fig12}. As above, $f(\sigma)$ is denoted by $\sigma'$.

\begin{lemma}
\label{simplyconnected}
The  simplicial complex ${\mathcal C}$  is simply connected.
\end{lemma}

\begin{proof}

Let $C$ be a cycle and let $f_0 : C \to \cc$ be a simplicial map.
We need to show that $f_0$ is null-homotopic.
For each $i \in [n]$, we define $f_i : C \to \cc$ by
$f_i(v) = shift_i(f_{i-1}(v))$ for every $v \in V(C)$.
Then by  Corollary~\ref{cor:simplicial:map} $f_0, \ldots, f_n$ are all homotopic to each other.
But $f_n(v) = I$ for every $v \in V(C)$.
This means that  $f_0, \ldots, f_n$ are all null-homotopic.


\begin{figure}[h!]
\begin{center}
\label{shell}
\includegraphics[scale=.3]{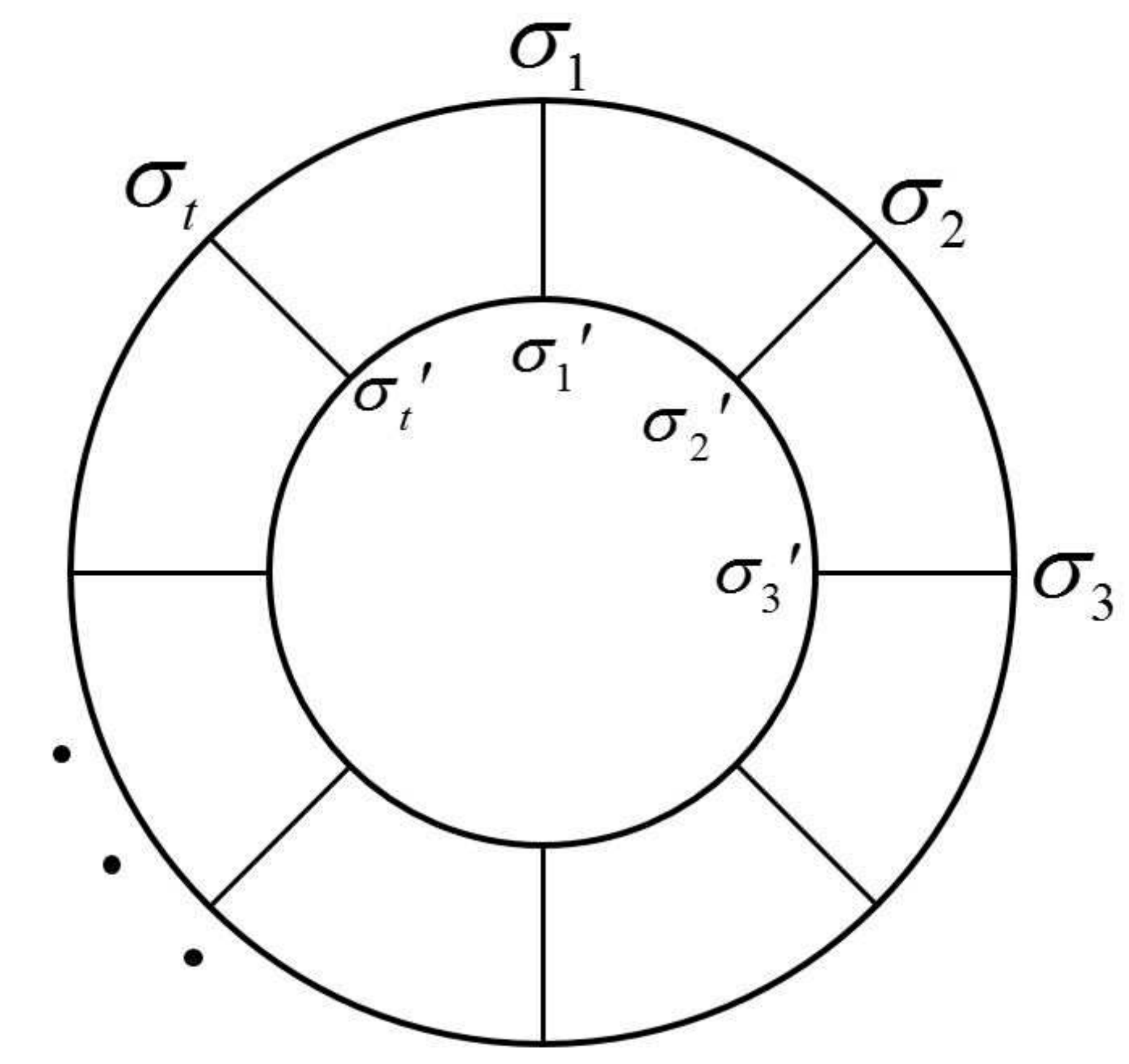}
\end{center}
\caption{}
\label{fig12}
\end{figure}

\end{proof}

\subsection{Associating a complex with the graph}

\begin{lemma}\label{lopsided3}
Let the set $E$ of edges of $K_{n, n}$ be partitioned to three sets
$E = E_1 \dot{\cup} E_2 \dot{\cup} E_3$.
Then there exists a perfect matching $M$ with at least $\ceil*{ \frac{|E_1|}{n} }$
edges of $E_1$ and at most $\ceil*{ \frac{|E_3|}{n} }$ edges of $E_3$.
\end{lemma}

\begin{proof}
Let $H$ be the graph with the edge set $E_1 \cup E_2$.
K\"onig's edge coloring theorem, combined with an easy alternating paths argument,
yields that  $H$ can be edge colored with $n$ colors in a way that each color class is of size
either $\left\lfloor \frac{|E(H)|}{n} \right\rfloor$ or $\left\lceil \frac{|E(H)|}{n} \right\rceil$. Clearly, at least one of these classes
contains at least $\frac{|E_1|}{n}$ edges from $E_1$. A matching with the desired property can be obtained by completing this color class in any way we please to a perfect matching of $K_{n, n}$.
\end{proof}

In fact,  a stronger property may hold:

\begin{conjecture}
\label{lopsided}
Let  $G=(V,E)$ be a bipartite graph with maximal degree $\Delta$ and let $f : E \to \{1,2,3,\ldots, k\}$ for some positive integer $k$.
Then there exists a matching $M$ in $G$ such that every number $j \in \{1,2,3,\ldots, k\}$ satisfies

$$ |\{e \in M \mid f(e) \leq j\}| \geq \left\lfloor \frac{|\{e \in E : f(e) \leq j\}|}{\Delta} \right\rfloor $$
\end{conjecture}

Clearly, we only need to see to it that the condition holds for  $j<k$.

In \cite{BLS} this conjecture was proved for $G=K_{6,6}$.\\

We shall say that a perfect matching $F$ has property $i^{(+)}$  if $|F \cap E_i| \ge k_i$,
 property $i^{(++)}$  if $|F \cap E_i| > k_i$,
and  property $i^{(-)}$  if $|F \cap E_i| \le k_i$.

\begin{figure}[h!]
\begin{center}
\label{triangle}
\includegraphics[scale=0.3]{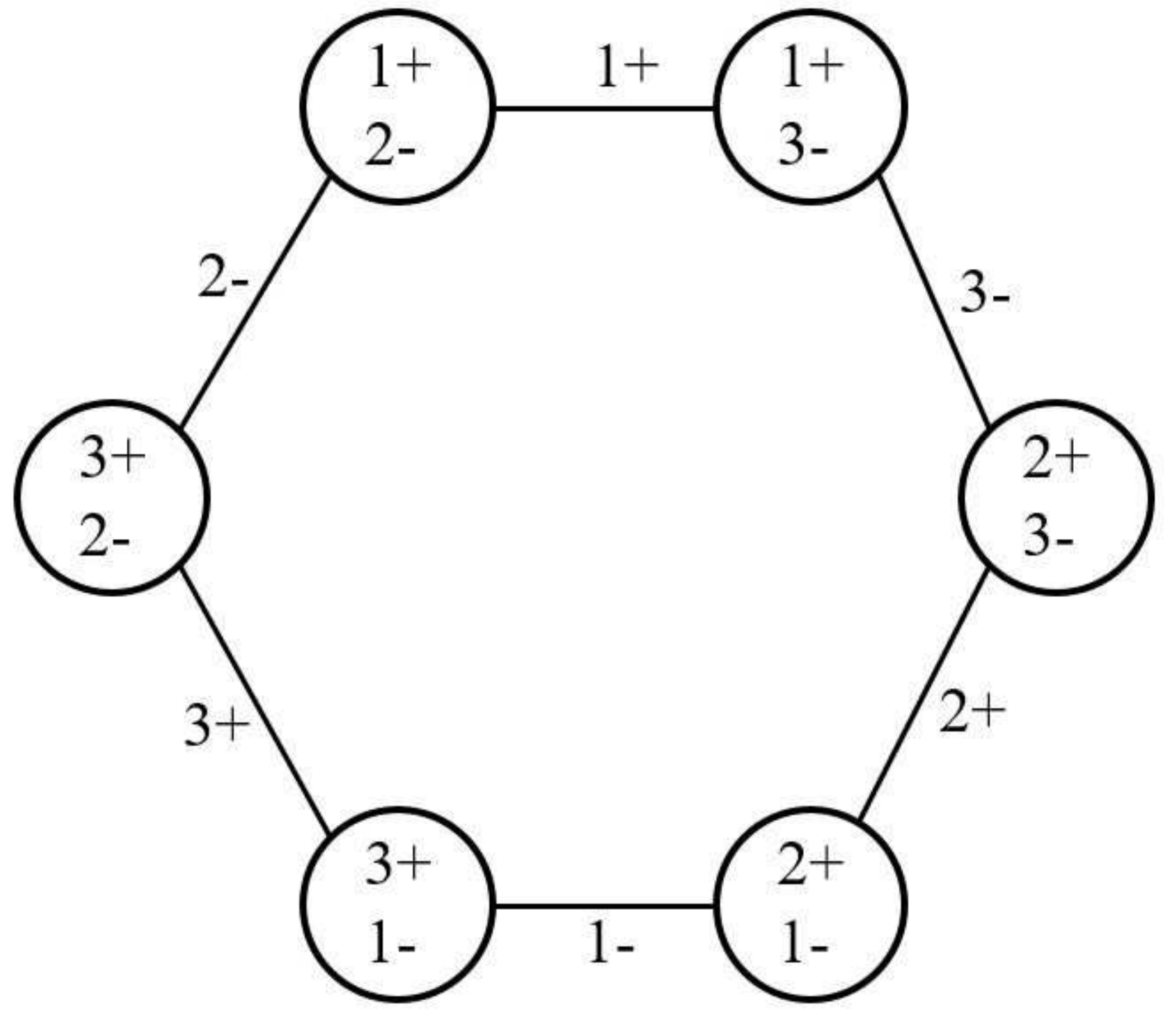}
\end{center}
\caption{}
\label{fig13}
\end{figure}

\begin{lemma}
\label{hexagonboundary}
There exists a triangulation of the boundary of a hexagon,
and an assignment of a perfect matching $M_v$ and a color $i_v \in \{1,2,3\}$ to each vertex $v$ of the triangulation, such that
$M_v$ has property $i_v^{(++)}$ and the coloring satisfies the conditions of Lemma \ref{hexagonsperner}.
\end{lemma}

\begin{proof}
By Lemma \ref{lopsided3} there exists a perfect matching $M$ with properties $1^{(+)}$ and $3^{(-)}$.
We assign it to one vertex of the hexagon.
By permuting the roles of $E_1, E_2, E_3$ we can find six such perfect matchings and assign them to the six vertices of the hexagon as in Figure~\ref{fig13}.

By Corollary \ref{path}, we can fill the path between the two permutations with property $i^{(-)}$
in a way that all perfect matchings in the path have property $i^{(-)}$. Similarly, we can fill the path between the two permutations with property $i^{(+)}$.
For each vertex $v$ we assign a color $i_v$ such that $M_v$ has property $i_v^{(++)}$.
If Lemma \ref{hexagonsperner} does not hold, then without loss of generality
we have two perfect matchings $M_1 \sim M_2$, where $M_1$ has properties $3^{(+)}$ and $1^{(++)}$ and $M_2$ has properties
$3^{(+)}$ and $2^{(++)}$. This yields Lemma \ref{hexagonboundary}.
\end{proof}



Since $\cc$ is simply connected, we can extend the mapping we got in Lemma \ref{hexagonboundary} to a triangulation of the hexagon.
Applying Lemma \ref{hexagonsperner} we obtain a triangle in the triangulation whose vertices are colored 1, 2 and 3.
This means that there exist $\sigma_1, \sigma_2, \sigma_3 \in S_n$, pairwise $\sim$ related and fairly representing
$E_1, E_2, E_3$ respectively.

\subsection{Proof of Theorem \ref{equirep=3}}

We form a matrix $A = (a_{i j })_{i,j \le n}$, where $a_{ij}=p~~(p=1,2,3)$ if the edge $ij$ belongs to $E_p$. \\


For each $\ell \in \{1,2,3\}$
and $\sigma \in S_n$ we write $d_\ell(\sigma) = |\{i : a_{i \sigma(i)} = \ell\}| - k_\ell$.\\

\begin{lemma}
Suppose that the triple $\{\sigma_1, \sigma_2, \sigma_3\}$ is in $\cc$,
and that $d_\ell(\sigma_\ell) > 0$ for each $\ell \in \{1,2,3\}$.
Then there exists $\sigma \in S_n$ with $|d_\ell(\sigma)| \leq 1$ for each $\ell \in \{1,2,3\}$.
\end{lemma}

Since the existence of such $\sigma_1, \sigma_2, \sigma_3$ follows from
 Lemmas \ref{hexagonsperner}, \ref{hexagonboundary} and \ref{simplyconnected}, this will finish the proof of Theorem
\ref{equirep=3}.

\begin{proof}
Define a   $3 \times 3$ matrix $B = (b_{i j})$ by $b_{i j} = d_i(\sigma_j)$.
We know that the diagonal entries in $B$ are positive, the sum in each column is zero,
and any two entries in the same row differ by at most 3.
This means that the minimal possible entry in $B$ is -2.
We may assume each column has some entry not in $\{-1, 0, 1\}$.

Let us start with the case that all of the diagonal entries of $B$ are at least 2.
This implies that all off-diagonal entries are at least -1.
Since each column must sum up to zero, we must have

$$B = \left(
\begin{array}{rrr}
2 & -1 & -1 \\
-1 & 2 & -1 \\
-1 & -1 & 2
\end{array}
\right)$$

This implie that the distance between any two of $\sigma_1, \sigma_2, \sigma_3$ is exactly 3, and without loss of generality
$\sigma_1 = I$, $\sigma_2 = (1 2 3)$, $\sigma_3 = (1 3 2)$, and the matrix $A$ has the form

$$A = \left(
\begin{array}{cccccc}
1 & 2 & 3 & * & \ldots & *\\
3 & 1 & 2 & * & \ldots & *\\
2 & 3 & 1 & * & \ldots & *\\
* & * & * & * & \ldots & *\\
\vdots & \vdots & \vdots & \vdots & \ddots & \vdots \\
* & * & * & * & \ldots & *
\end{array}
\right)$$

We can now take $\sigma = (1 2)$ and we are done.



We are left with the case that some diagonal entry of $B$ is 1.
Without loss of generality $b_{1 1} = 1$.
We also assume without loss of generality that $b_{2 1} \leq b_{3 1}$.
Since the first column must sum up to zero, we have
$b_{2 1} + b_{3 1} = -1$, and thus
$-0.5 = 0.5(b_{2 1} + b_{3 1}) \leq b_{31}  = -1 - b_{21} \leq 1$.
In other words, either $b_{2 1} = -1$ and $b_{3 1} = 0$
or $b_{2 1} = -2$ and $b_{3 1} = 1$.
In the first case we can just take $\sigma = \sigma_1$ and we are done.
Therefore we assume the second case.

$$B = \left(
\begin{array}{rrr}
1 & * & * \\
-2 & * & * \\
1 & * & *
\end{array}
\right)$$

Since $d_3(\sigma_1)> 0$, we may assume $\sigma_3 = \sigma_1$, and due to the -2 entries in the second row,
we must have $b_{2 2} = 1$.  We now get

$$B = \left(
\begin{array}{rrr}
1 & * & 1 \\
-2 & 1 & -2 \\
1 & * & 1
\end{array}
\right)$$

Without loss of generality $b_{1 2} \leq b_{3 2}$ and by arguments similar to the above we can fill the second column

$$B = \left(
\begin{array}{rrr}
1 & -2 & 1 \\
-2 & 1 & -2 \\
1 & 1 & 1
\end{array}
\right)$$







The distance between $\sigma_1$ and $\sigma_2$
is exactly 3, so
without loss of generality $\sigma_1 = I$ and $\sigma_2 = (1 2 3)$.
In order to achive the values of $b_{1 2} = -2, b_{1 1} = 1, b_{2 1} = -2, b_{2 2} = 1$ we must have
$a_{i i} = 1$ and $a_{i \sigma_2(i)} = 2$ for each $i \in \{1,2,3\}$.

The only case in which none of the choices $\sigma = (1 2)$
or $\sigma = (2 3)$ or $\sigma = (1 3)$ works is if
$a_{1 3} = a_{2 1} = a_{3 2} = 3$, so once again we get

$$A = \left(
\begin{array}{cccccc}
1 & 2 & 3 & * & \ldots & *\\
3 & 1 & 2 & * & \ldots & *\\
2 & 3 & 1 & * & \ldots & *\\
* & * & * & * & \ldots & *\\
\vdots & \vdots & \vdots & \vdots & \ddots & \vdots \\
* & * & * & * & \ldots & *
\end{array}
\right)$$



We have $b_{3 1} = 1$ which means that 3 appears $k_3+1$ times on the diagonal.
Without loss of generality $a_{4 4} = a_{5 5} = \ldots = a_{k_3+4 \, k_3+4} = 3$.
In any of the following cases one can easily find some $\sigma \in S_n$
with $|d_\ell(\sigma)| \leq 1$ for each $\ell \in \{1,2,3\}$:
\begin{itemize}
\item If either $a_{i j} \neq 3$ or $a_{j i} \neq 3$ for some $i \in \{4, \ldots, k_3+4\}$ and $j \in \{1,2,3\}$.
\item If $a_{i j} \neq 3$ for some $i, j \in \{4, \ldots, k_3+4\}$
\item  If both $a_{i j} \neq 3$ and $a_{j i} \neq 3$ for some $i \in \{4, \ldots, k_3+4\}$ and $j \in \{ k_3+5, \ldots, n\}$.
\end{itemize}

If none of the above occurs then
$$k_3 n = |\{(i,j) : a_{i j} = 3\}| \geq 2 \cdot 3 \cdot (1+k_3) + (1+k_3)^2 + \frac{1}{2} \cdot 2 (k_3+1)(n-k_3-4) $$
which is a contradiction.
\end{proof}

\begin{remark}
After the above topological proof of Theorem \ref{equirep=3} was found,
a combinatorial proof was given in \cite{BLS}.
\end{remark}

{\bf Acknowledgement} The authors are grateful to Fr\'ed\'eric Meunier for pointing out an inaccuracy in a previous version of the paper.

\end{document}